\documentclass[11pt,reqno]{amsart}
\usepackage{amssymb}
\usepackage{amsfonts}
\usepackage{mathrsfs}
\usepackage{amsmath}
\usepackage{graphicx}
\usepackage{hyperref}
\usepackage{float}
\usepackage{epstopdf}
\usepackage{color}
\usepackage{bm}
\usepackage{comment}
\usepackage{soul}
\usepackage{enumerate}

\allowdisplaybreaks
\newtheorem{theorem}{Theorem}[section]
\newtheorem{proposition}[theorem]{Proposition}
\newtheorem{lemma}[theorem]{Lemma}
\newtheorem{corollary}[theorem]{Corollary}
\newtheorem{remark}[theorem]{Remark}

\numberwithin{equation}{section}

\begin{document}
\title{Assouad and lower dimensions of graph-directed  Bedford-McMullen carpets}

\author{Hua Qiu}
\address{School of Mathematics, Nanjing University, Nanjing, 210093, P. R. China.}
\thanks{The research of Qiu was supported by the National Natural Science Foundation of China, grant 12470187}
\email{huaqiu@nju.edu.cn}

\author{Qi Wang}
\address{School of Mathematics, Nanjing University, Nanjing, 210093, P. R. China.}
\email{602023210013@smail.nju.edu.cn}
\author{Shufang Wang$^*$}
\thanks{*Corresponding author}
\address{School of Mathematics, Nanjing University, Nanjing, 210093, P. R. China.}
\thanks{}
\email{wangsf6822@yeah.net}

%    General info

\date{}


\maketitle

\begin{abstract}
We calculate the Assouad and lower dimensions of graph-directed Bedford-McMullen carpets, which reflect the extreme local scaling laws of the sets, in contrasting with known results on Hausdorff and box dimensions. We also investigate the relationship between distinct dimensions. In particular, we identify an equivalent condition when the box  and Assouad dimension coincide, and show that under this condition, the Hausdorff dimension attains the same value.

\end{abstract}

\section{Introduction}

Since the mid 80s, the study of self-affine sets has emerged as an independent research field, in which the class of Bedford-McMullen carpets \cite{B84, Mc84} plays a prominent role. There has been significant interest in calculating the dimensions of these sets as well as their generalizations, including the Lalley-Gatzouras class \cite{LG92},  Bara\'{n}ski carpets \cite{Bar07}, Feng-Wang box-like sets \cite{FW05, F12,F16}, Kenyon-Peres $(\times m,\times n)$ carpets generated by shifts of finite type or sofic shifts \cite{KP96, F23}, and their high-dimensional analogs \cite{DS17, FJ21, K23,Olsen98}. Unlike the case of self-similar sets, a key feature in this study is the provision of plenty of examples generated by recursive constructions, which exhibit distinct Hausdorff and box dimensions.

In 2011, Mackay \cite{M11} initiated the study of Assouad dimension of sets generated by Lalley-Gatzouras construction, including Bedford-McMullen carpets. This investigation was later extended by Fraser \cite{F14} to Bara\'{n}ski carpets,  which are generated by a box-like affine construction with a more flexible grid structure. As a natural dual, the exact value of the lower dimension (introduced by Larman  \cite{Lar67}) of sets in this scenario was also determined in  \cite{F14}. As a rapidly expanding branch of dimension theory on fractals, Assouad and lower dimensions focus on the local geometric information of fractals, reflecting the scaling laws of the thickest or thinnest parts of the sets. For further background and details on these two dimensions, refer to \cite{F21}.
 A fascinating observation in \cite{M11, F14} for the Lalley-Gatzouras class is the dichotomy where the Hausdorff, box, Assouad and lower dimensions are either all distinct or  equal.

In this note, we extend the consideration of Assouad and lower dimensions to graph-directed Bedford-McMullen carpets, which serves as a typical example of $(\times m,\times n)$ carpets generated by shifts of finite type introduced in Kenyon-Peres \cite{KP96}.

\vspace{0.2cm}

Let us begin with a \textit{finite directed graph}, denoted as $G:=(V,E)$, where $V$ represents the \textit{vertex set} and $E$ denotes the \textit{directed edge set}, allowing loops and multiple edges. For each edge $e$ in $E$, let $i(e)$ be the \textit{initial vertex} of $e$, $t(e)$ be the \textit{terminal vertex} of $e$, and denote the edge as $v \stackrel{e}{\rightarrow}v'$. We always assume that for each $v\in V$, there exists an edge $e\in E$ such that $i(e)=v$.
Given two positive integers $n>m$, we associate each edge $e$ in $E$ with an affine contraction $\psi_e:[0,1]^2\to [0,1]^2$ given by:
$$
\psi_e\left(\begin{aligned}
    &\xi_1 \\&\xi_2
\end{aligned}\right)= \left(
\begin{array}{cc}
    n^{-1} & 0\\
    0 & m^{-1}
\end{array} \right) \left( \begin{aligned}
    &\xi_1+x_e \\&\xi_2 +y_e
\end{aligned}\right), \quad \text{ for } (\xi_1, \xi_2) \text { in }{[0,1]^2},
$$
where $(x_e,y_e)\in \{0,1,\cdots,n-1\}\times\{0,1,\cdots,m-1\}$. Let $\Psi=\{\psi_e\}_{e\in E}$ be the collection of all these contractions. Then the triple $(V,E,\Psi)$ becomes a \textit{graph-directed iterated function system}. It is well known that there exists a unique family of attractors $\{ X_v\}_{v\in V}$ contained in $[0,1]^2$ satisfying
$$
X_v =\bigcup_{e\in E:i(e)=v} \psi_e(X_{t(e)}), \quad \text{ for all }v\in V.
$$
We refer to $\{X_v\}_{v\in V}$ as a \textit{graph-directed Bedford-McMullen  $(\times m,\times n)$-carpet family} generated by $(V,E,\Psi)$ and  $X:=\bigcup_{v\in V}X_v$ as a \textit{graph-directed Bedford-McMullen  $(\times m,\times n)$-carpet}.

 Now, let us define $$
E^\infty:=\{ \omega=\omega_1 \omega_2 \cdots:  \omega_{i}\in E, t(\omega_{i})=i(\omega_{i+1}) \text{ for all } i\geq 1 \}.
$$
as the collection of \textit{infinite admissible words} along $G$.
For $\omega\in E^\infty$, denote $i(\omega):=i(\omega_1).$ It is straightforward to observe that
$$
X_v=\left\{\left(\sum_{i=1}^{\infty}\frac{x_{\omega_i }}{n^i},\sum_{i=1}^{\infty}\frac{y_{\omega_i}}{m^i}\right): \omega\in E^\infty, i(\omega)=v \right\}, \quad \text{ for all  } v\in V.
$$

%In the following, we consider $X=\bigcup_{v\in V}X_v$ and we call it the \textit{graph-directed carpet}.

Throughout this note, in accordance with standard conventions, we will write
$\dim_H$, $\underline{\dim}_B$, $\overline{\dim}_B$, $\dim_A$, $\dim_L$ for the Hausdorff, lower box, upper box, Assouad and lower dimensions, respectively. If $\underline{\dim}_B$
 and $\overline{\dim}_B$  coincide, we simply refer to the box dimension, denoted as $\dim_B$. For further details on these  dimensions, refer to \cite{Fal14,F21}.

\vspace{0.2cm}

Denote $\pi:\mathbb{R}^{2}\rightarrow \mathbb{R}$ the projection map onto the second coordinate axis, i.e. $\pi(\xi_1,\xi_2)=\xi_2$.  We will introduce two sequences $\{\alpha_k\}_{k\geq 1}$, $\{\beta_k\}_{k\geq 1}$ whose values depend on  $\{\dim_B\pi(X_v)\}_{v\in V}$, see details in Subsection \ref{subsec22}. Taking $\alpha=\lim_{k\to \infty}(\alpha_k)^{1/k}$ and $ \beta=\liminf_{k\rightarrow \infty} (\beta_k)^{1/k},$ we will prove that

\begin{theorem}\label{th1}
    Let  $\{X_v\}_{v\in V}$ be a   graph-directed Bedford-McMullen $(\times m, \times n)$-carpet family. Write $X=\bigcup_{v\in V} X_v$.  We have
    \begin{equation}\label{s1}
    \dim_A X =\frac{\log \alpha}{\log n}\quad\text{ and }\quad \dim_L X =\frac{\log \beta}{\log n}.
    \end{equation}

\end{theorem}

The result in Theorem \ref{th1} can be seen as a complementary to the Hausdorff and box dimensions of $X$ considered in \cite{F23,KP96}, noting that Assouad and lower dimensions are more sensitive to the extreme local structure of fractals, whereas Hausdorff and box dimensions reveal more global geometric information.

It is worth noting that unlike previous works on $(\times m, \times n)$ carpets, we do not assume a separating condition for $(V, E, \Psi)$, i.e. we allow two distinct contractions in $\Psi$ to have the same box-like images. Recall that in the usual consideration \cite{B84, KP96,Mc84}, the  \textit{rectangular open set condition}(ROSC) is always assumed,
$$
\bigcup_{e\in E:i(e)=v}\psi_e((0,1)^2) \subseteq(0,1)^2, \quad \text{ for all } v\in V,
$$
    and the union is disjoint (See also \cite{FW05}). Benefiting from the lattice structure of Bedford-McMullen carpets, we can drop ROSC, which is necessary since on $X$, the union set of carpet family, complete overlap of boxes inevitably happens in general.

    We should also point out that once Theorem \ref{th1} is established, we can calculated the exact values of dimensions for each $X_v$, $v\in V$, since $\dim_A X_v=\dim_A \bigcup_{v'\in V_v} X_{v'}$ noticing that  $\{X_{v'}\}_{v'\in V_v}$ is a carpet family generated by $(V_v, E_v, \Psi_v)$ with $V_v:=\{v'\in V: v\rightarrow\cdots\rightarrow v'\}\cup\{v\}$, $E_v:=\{e\in E: i(e)\in V_v, t(e)\in V_v\}$ and $\Psi_v:=\{\psi_e\in \Psi: e\in E_v\}$.

    Another noteworthy aspect of this theorem is that when proving the lower bound of Assouad dimension and the upper bound of lower dimension, we avoid using the technique of constructing \textit{weak tangents} as done in \cite{M11, F14}, since achieving it would be challenging when dealing with  more complicated fractals.
    %When ROSC hold, the Hausdorff and box dimensions of $X$ have known by \cite[Theorem 1.1]{KP96} and \cite[Theorem 1.4]{QWW23}. Actually, a partial result of box dimension could be seen in \cite[Proposition 1.2]{KP96}, which only require the adjacency matrix is primitive $A$, i.e. some power of $A$ has positive entries.

   \vspace{0.2cm}
    To make a quick comparison of our result with those of Mackay \cite{M11} and Fraser \cite{F14}, let us define $A$ as the $\#V\times \# V$ \textit{adjacency matrix} of the directed graph $G=(V,E)$, where $A(v,v')$ represents the number of edges in $E$ from $v$ to $v'$ for $v,v'$ in $V$; and for $0\leq j<m$, let $A_j$ be a $\#V \times \#V$ matrix defined as
    $$
    A_j(v,v')=\# \{ e\in E: v \stackrel{e}{\rightarrow}v', y_e=j\}, \quad \text{ for all }v,v'\in V.
    $$ Note that here
    $
     A=\sum_{j=0}^{m-1} A_j.
    $
    The following corollary suggests that, similar to \cite{M11, F14}, in the graph-directed setting, the Assouad dimension of $X$ should be still the summation of $\dim_B \pi(X)$ and the maximal dimension of slices vertical to $\pi(X)$. Precisely,
    \begin{corollary}\label{c1}
    Let $X$ be same as in Theorem \ref{th1}. Assume that $ROSC$ holds and $A$ is \emph{irreducible,} i.e. $\sum_{j\geq 1} A^j$ is positive. We have
    $$
    \begin{aligned}
    \dim_A X&=\dim_B \pi(X)+\frac{1}{\log n} \lim_{k\to \infty} \frac{1}{k}\log \max_{y_1,\cdots, y_k\in \{0,1,\cdots,m-1\}}  \Vert A_{y_1} \cdots A_{y_k} \Vert,
    \end{aligned}
    $$
    where $\Vert\cdot\Vert$ represents any matrix norm.
\end{corollary}
   We will extend this observation to the general case in subsequent Theorem \ref{th3}.
    \vspace{0.2cm}

    The second main aim of this note is to explore whether  a dichotomy similar to \cite{M11, F14} exists, in that the values of distinct dimensions are either all distinct or equal. To this end, we establish an equivalent  condition (see \eqref{e26}) under which the box and Assouad dimensions of $X$ coincide. Moreover, we prove that under this condition, it further follows that
    $$\dim_HX=\dim_BX=\dim_AX.$$ However, there is no a general dichotomy as wondered, since we can provide an example of $X$ for which it holds that
    $$\dim_L X<\dim_H X=\dim_B X<\dim_A X.$$

    To illustrate this, let us introduce some additional notation. Let $$\mathcal{I}=\{(x_{\omega_1},y_{\omega_1})(x_{\omega_2},y_{\omega_2})\cdots:\omega \in E^\infty \} \subseteq \big \{ \{0,\cdots, n-1\}\times \{0,\cdots, m-1\} \big \}^{\mathbb{N}}
    $$
    and denote
    $$
    \pi \mathcal{I}=\{y_{\omega_1} y_{\omega_2}\cdots:\omega\in E^\infty\} \subseteq \{0,\cdots, m-1\}^{\mathbb{N}}.
    $$
    Here, somewhat abusing the notation, we continue to use $\pi$ to represent the `\textit{projection}' from $\mathcal{I}$ onto $\pi \mathcal{I}$, i.e. we write
    $
    \pi((x_{\omega_1},y_{\omega_1})(x_{\omega_2},y_{\omega_2})\cdots)=y_{\omega_1} y_{\omega_2}\cdots.
    $

    Let $\sigma:\mathcal{I}\to \mathcal{I}$ (resp. $\tilde{\sigma}:\pi \mathcal{I}\to \pi\mathcal{I}$) denote the left shift map. Clearly, $\sigma(\mathcal{I})\subseteq \mathcal{I}$ and $\tilde{\sigma}(\pi \mathcal{I})\subseteq \pi \mathcal{I}$. Regard $(\mathcal{I},\sigma)$ and $(\pi\mathcal{I},\tilde{\sigma})$ as two topological dynamics and $\pi:(\mathcal{I},\sigma)\to (\pi \mathcal{I},\tilde{\sigma})$ as a projection  satisfying $\tilde{\sigma}\circ \pi =\pi \circ \sigma$.
 In a standard manner, we use $h_{\textup{top}}(K)$ to denote the \textit{topological entropy}  of a compact  set $K$ in $\mathcal{I}$ (resp. $\pi{\mathcal{I}}$) (In our context, we always omit $\sigma$ (resp. $\tilde{\sigma}$)).

Return to the directed graph $G=(V,E)$. If the adjacency matrix $A$ of $G$ is irreducible, we say $G$ is an \textit{irreducible directed graph}. An \textit{irreducible component} $H$ of $G$ is a maximal subgraph  of $G$ such that $H$ is irreducible.  Let $\{H_i:=(V_i,E_i)\}_{i=1}^r$ denote all the irreducible components of $G$. For $i=1,\cdots,r,$ denote $$\mathcal{I}_{H_i}=\{(x_{\omega_1},y_{\omega_1})(x_{\omega_2},y_{\omega_2})\cdots:\omega \in E_i^\infty \}\subseteq \mathcal{I},$$ where $E_i^\infty$ represents the collection of infinite admissible words along $H_i$. Accordingly, write $\pi\mathcal{I}_{H_i}=\pi(\mathcal{I}_{H_i})$. Let $\{i\}^+$ be the collection of $1\leq j\leq r$ such that there is a path from a vertex in $H_i$ to a vertex in  $H_j$, i.e. $v\to \cdots \to v'$ for some $v$ in $V_i$, $v'$ in $V_j$.

Recently, Fraser and Jurga \cite[Theorem 1.2]{F23} derived the following formula for the box dimension of $X$,
\begin{equation}\label{e25}
\dim_B X= \max_{1\leq i\leq r}\left \{ \frac{h_{\textup{top}}(\mathcal{I}_{H_i})}{\log n}+ \max_{j\in \{i\}^+} h_{\textup{top}}(\pi \mathcal{I}_{H_j})\left (\frac{1}{\log m}-\frac{1}{\log n}\right) \right \}.
\end{equation}
Roughly, it is the maximal box dimension of carpets among all irreducible components.

Our second main result is as follows.

\begin{theorem}\label{th3}
    Let $X$ be same as in Theorem \ref{th1}, we have
    $$
    \dim_A X=\max_{1\leq i\leq r}\left\{ \sup_{y\in \pi \mathcal{I}_{H_i}} \frac{h_{\textup{top}}(\pi^{-1}(y)\cap \mathcal{I}_{H_i})}{\log n}+\max_{j\in \{i\}^+}\frac{h_{\textup{top}}(\pi\mathcal{I}_{H_j})}{\log m}  \right\}.
    $$
   In particular, $\dim_B X=\dim_A X$ if and only if there exists $i\in \{1,\cdots, r\}$ such that
    \begin{equation}\label{e26}
    \begin{aligned}
    \dim_B X&= \frac{h_{\textup{top}}(\mathcal{I}_{H_i})}{\log n}+  h_{\textup{top}}(\pi \mathcal{I}_{H_i})\left (\frac{1}{\log m}-\frac{1}{\log n}\right),
    \\
    \dim_A X&= \sup_{y\in \pi \mathcal{I}_{H_i}} \frac{h_{\textup{top}}(\pi^{-1}(y)\cap \mathcal{I}_{H_i})}{\log n}+\frac{h_{\textup{top}}(\pi\mathcal{I}_{H_i})}{\log m},
    \\
    h_{\textup{top}}(\mathcal{I}_{H_i})&=\sup_{y\in \pi \mathcal{I}_{H_i}} h_{\textup{top}}(\pi^{-1}(y)\cap \mathcal{I}_{H_i})+h_{\textup{top}}(\pi\mathcal{I}_{H_i}).
    \end{aligned}
    \end{equation}
    Furthermore, if $\dim_B X=\dim_A X$, then
    $$
    \dim_H X =\dim_B X =\dim_A X.
    $$
\end{theorem}

\begin{theorem}\label{th4}
    When $G$ is an irreducible directed graph, for any graph-directed Bedford-McMullen carpet $X$, we have 
    $$\dim_H X=\dim_B X \quad\emph{ implies }\quad  \dim_B X=\dim_A X.$$ 
    When $G$ is not irreducible, there is a graph-directed Bedford-McMullen carpet $X$ such that
    $$
    \dim_H X =\dim_B X<\dim_A X.
    $$
\end{theorem}

The organization of this note is as follows. In Section \ref{sec2}, we introduce necessary notations and define the numbers $\alpha$ and $\beta$. In Section \ref{sec3}, we prove Theorem \ref{th1} as well as Corollary \ref{c1} for the Assouad dimension. In Section \ref{sec4}, we prove Theorem \ref{th1}  for the lower dimension. Finally, in Section \ref{sec5}, we prove Theorems \ref{th3} and \ref{th4} for the relationship between distinct dimensions.
\vspace{0.2cm}

\section{Preliminary}\label{sec2}

Let  $\{X_v\}_{v\in V}$ be same as in Theorem \ref{th1}. In this section, we mainly introduce two numbers $\alpha, \beta$ which appear in the expression of Assouad and lower dimensions.

\subsection{Notations and background.}

Firstly, let us introduce some notations. Denote
$$
E^*:=\{ w=w_1\cdots w_k:\omega\in E^\infty, w_i=\omega_i \text{ for all } 1\leq i\leq k, k\in \mathbb{N} \} \cup \{\emptyset\}
$$
the collection of  \textit{ finite admissible words}.
For $w=w_1\cdots w_k \in E^*$, denote $|w|=k$ the \textit{length} of $w$, $i(w):=i(w_1)$, $t(w):=t(w_k)$ the \textit{initial} and \textit{terminal} vertices of $w$, and write $i(w) \stackrel{w}{\to} t(w).$ For $l\leq k$, denote $w|_l=w_1\cdots w_l$. Write
$$
\psi_w=\psi_{w_1}\circ \psi_{w_2}\circ \cdots \circ \psi_{w_k},\quad x_w=x_{w_1}x_{w_2}\cdots x_{w_k},\quad y_{w}=y_{w_1}y_{w_2}\cdots y_{w_k}
$$
for short and denote $\psi_{\emptyset}=id$ (the identity map) by convention. Denote $E^k$ the collection of admissible words of \textit{length} $k$. For $w,w'\in E^*$ with $i(w')=t(w),$ write $ww'$ the \textit{concatenation} of $w$ and $w'.$

For  $v,v'\in V$, we say there exists a \textit{directed path} from $v$ to $v'$ if there exists $w\in E^*$  satisfying $v \stackrel{w}{\rightarrow}v'$ (write simply $v {\rightarrow}v'$ when we do not emphasize $w$). Note that $G=(V,E)$ is irreducible if and only if $v\to v'$ for each distinct pair $v,v'\in V$.

For $w,w'\in E^*$, we write $w\sim w'$ if and only if $\psi_w=\psi_{w'}$. Denote $[w]=\{w'\in E^*:w'\sim w\}$ the equivalence class of $w$. We  use $[E^*]$ to denote the collection of all equivalence classes with respect to `$\sim$', i.e.
$$
    [E^*]=\{[w]:w\in E^*\}.
$$
In a same way, denote $[E^k]=\{[w]:w\in E^k \}$. Denote $t(v,[w])=\{t(w'):w'\in [w],i(w')=v\}$ the collection of terminal vertices of $[w]$ from $v$ and write $t([w])=\bigcup_{v\in V} t(v,[w])$.

Write  $\Sigma_X=\{0,1,\cdots, n-1\}$ and $\Sigma_Y=\{0,1,\cdots, m-1\}$, the \textit{alphabets} along the first and second coordinate axis, respectively.  Denote $\Sigma_X^k$ (resp. $\Sigma_Y^k$) the collection of corresponding words of length $k$, and write $\Sigma_X^0$ (resp. $\Sigma_Y^0$) $=\{\emptyset\}$ for convenience. Let $\Sigma^*_X=\bigcup_{k\geq 0} \Sigma_X^k$ and  $\Sigma^*_Y=\bigcup_{k\geq 0} \Sigma_Y^k$. Clearly, for $w\in E^k$, $x_w\in \Sigma_X^k$, $y_w\in \Sigma_Y^k$.

For $e\in E$, define
$$
\phi_e(\cdot):=\frac{1}{m}(\cdot+y_e).
$$
It is easy to see that $\phi_w=\phi_{w'}$  for $w,w'\in E^k$ with $y_w=y_{w'}$.

\begin{proposition}\label{p1}
    For each $v\in V$,
    the box dimension of $\pi(X_v)$ exists, and
    \begin{equation}\label{e3}
    \dim_H \pi(X_v) =\dim_B \pi(X_v).
    \end{equation}
    In particular, if $G$ is irreducible, we have
    \begin{equation}\label{e1}
    \dim_H \pi(X_v)= \dim_B \pi(X_v)=\dim_H \pi(X_{v'})= \dim_B \pi(X_{v'}), \quad \text{ for } v, v'\in V.
    \end{equation}
\end{proposition}
\begin{proof}
   Note that $\{\pi(X_v)\}_{v\in V}$ is a family of graph-directed self-similar sets generated by a graph-directed iterated function system  $(V,E,\Phi)$, where
    \begin{equation}
    \Phi=\{\phi_e:e\in E\}. \nonumber
    \end{equation}
    By \cite[Theorem 2.7]{DN04}, $(V,E,\Phi)$ satisfies a  finite type overlapping condition. Combining this with \cite[Theorem 1.1]{DN04}, we see that  $\dim_B \bigcup_{v\in V} \pi(X_v)$ exists and
    \begin{equation}\label{e2}
       \dim_B \bigcup_{v\in V} \pi(X_v) =\dim_H \bigcup_{v\in V} \pi(X_v).
    \end{equation}

    For each $v\in V$, if $v$ is a \textit{root} of $G$ (i.e. $v\to v'$ for all other vertices $v'$ in $V$), we have
    $$
    \dim_H \bigcup_{v'\in V} \pi(X_{v'}) = \dim_H \pi(X_v) \leq \overline{\dim}_B \pi(X_v) \leq {\dim}_B \bigcup_{v'\in V} \pi(X_{v'}).
    $$
   So, $\dim_B \pi(X_v)$ exists and \eqref{e3} holds,  by \eqref{e2}.
   \vspace{0.2cm}

   If $v$ is not a root of $G$, define
    $$
    V_v=\{v'\in V:v\to v'\}\cup \{v\}, \quad E_v=\{e\in E:i(e),t(e)\in V_v\}, \quad  \Phi_v=\{\phi_e\in \Phi:e\in E_v\}.
    $$
    Clearly, $\pi(X_v)$ belongs to the family of attractors generated by $(V_v,E_v,\Phi_v)$, and $v$ is a root of $(V_v,E_v)$. By a  same argument as above, we still get \eqref{e3}.

    If $G$ is irreducible, it is straightforward to see \eqref{e1}.
\end{proof}

\subsection{Two numbers $\alpha$ and $\beta$}\label{subsec22}

Let $ \eta: V\times \bigcup_{k\geq 0} (\Sigma_X^k\times \Sigma_Y^k) \to \{\dim_B\pi (X_v):v\in V\}\cup \{0\}$ be a function defined by
\begin{equation}\label{e12}
\eta(v,x,y)=\left \{ \begin{aligned}
    &\dim_B \pi\left(X_v\cap \psi_w((0,1)^2)\right) & &\text{if } i(w)=v, (x_w,y_w)=(x,y) \text{ for some }w\in E^*,
    \\
    &0 & &\text{otherwise}.
\end{aligned} \right.
\end{equation}
Clearly, $\eta(v,\emptyset,\emptyset)=\dim_B \pi(X_v)$, $\eta(v,x_w,y_w)=\eta(v,x_{w'},y_{w'})$ for $w\sim w'$, and
\begin{equation}\label{e17}
\eta(v,x_w,y_w) =\dim_B \bigcup_{v'\in t(v,[w])} \pi(X_{v'})=\max_{v'\in t(v,[w])} \dim_B \pi(X_{v'}).
\end{equation}
For $w,w'\in E^*$ with $t(w)=i(w')$, it is not hard to check that
\begin{equation}\label{e11}
\eta(i(w),x_w,y_w)\geq \eta({i(w)},x_{ww'},y_{ww'})=\max_{v\in t(i(w),[w])}\eta(v,x_{w'},y_{w'}).
\end{equation}

Define $$\tilde{V}=\{v\in V:\text{ there exists } w\in E^* \text{ with } |w|\geq \#V \text{ and } t(w)=v\}$$
the collection of vertices in $V$ which are terminal vertices of some admissible words with length at least $\#V$.

For $k\geq 1$, define
$$
\alpha_k:=\max_{v\in \tilde{V}}\max_{y \in \Sigma_Y^k} \sum_{[w]\in [E^k]:\atop i(w)=v, y_{w}=y} n^{k\eta(v,x_{w},y_{w})}.
$$
%Since the set $\{ t([w]):w\in E*\}$ is finite, the second maximum is well-defined.

\begin{lemma}\label{lim}
    The limit of $\{(\alpha_k)^{1/k}\}$ exists and equals $\inf_{k\geq 1} (\alpha^k)^{1/k}$.
\end{lemma}
\begin{proof}
    For $k,l$ in $\mathbb{N}$ with $k\geq \#V$, we have
    $$
    \begin{aligned}
    \alpha_{k+l}&=\max_{v\in \tilde{V}} \max_{y\in \Sigma_Y^{k+l}} \sum_{[w]\in [E^{k+l}]:\atop i(w)=v, y_{w}=y} n^{(k+l) \eta(v,x_{w},y_{w})}
    \\
    &\leq \max_{v\in \tilde{V}} \max_{y \in \Sigma_Y^k,y'\in \Sigma_Y^l} \sum_{[w]\in [E^k]:\atop i(w)=v,  y_{w}=y} n^{k\eta(v,x_{w},y_{w})}  \sum_{[w']\in [E^l]:\atop i(w')\in t(v,[w]), y_{w'}=y'}\max_{v'\in t(v,[w])} n^{l\eta(v',x_{w'},y_{w'})}
    \\
    &\leq \max_{v\in \tilde{V}} \max_{y \in \Sigma_Y^k,y'\in \Sigma_Y^l} \sum_{[w]\in [E^k]:\atop i(w)=v,  y_{w}=y} n^{k\eta(v,x_{w},y_{w})}\sum_{v'\in t(v,[w])} \sum_{[w']\in [E^l]:\atop i(w')=v' , y_{w'}=y'} n^{l\eta(v',x_{w'},y_{w'})}
    \\
    &\leq\#V \max_{v\in \tilde{V}} \max_{y \in \Sigma_Y^k}  \sum_{[w]\in [E^k]:\atop i(w)=v,  y_{w}=y} n^{k\eta(v,x_{w},y_{w})}    \max_{v'\in \tilde{V}} \max_{y' \in \Sigma_Y^l}  \sum_{[w']\in [E^l]:\atop i(w')=v', y_{w'}=y'} n^{l\eta(v',x_{w'},y_{w'})}
    \\
    &=\#V \alpha_k\alpha_l,
    \end{aligned}
    $$
    where the second line follows from \eqref{e11}, and the forth line follows from the fact that $t(v,[w])\subseteq \tilde{V}$ since $|w|\geq \#V$. Therefore, $\{\alpha_k\}_{k\geq 1}$ is a submultiplicative sequence and the lemma follows.
\end{proof}

Write
\begin{equation}
    \alpha=\lim_{k\to \infty}(\alpha_k)^{1/k}.\nonumber
\end{equation}
In Section \ref{sec3}, we will prove that the Assouad dimension of $X$ is $\frac{\log \alpha}{\log n}.$
\vspace{0.2cm}

Let
$
\theta: V\times \bigcup_{k\geq 0} (\Sigma_X^k\times \Sigma_Y^k )\times \Sigma_Y^* \to \{\dim_B\pi (X_v):v\in V\}\cup\{0\}
$ be a function  defined by
$$
\theta(v,x,y,y')=\left \{ \begin{aligned}
    &\eta(v,x,y) & &\text{if } y'=\emptyset,
    \\
    &
\dim_B \pi\left(X_v\cap \psi_w((0,1)^2)\right)\cap \phi_w(I_{y'}) & & \text{if } y'\neq \emptyset,
\\
&0 & & \text{otherwise},
\end{aligned} \right.
$$
where $w$ is same as that in the definition of $\eta$ in \eqref{e12} and
\begin{equation}\label{e18}
I_{y'}=\left(\sum_{i=1}^l \frac{y_i'}{m^i}, \sum_{i=1}^l \frac{y'_i}{m^i}+\frac{1}{m^l}\right),\quad  \text{ for }y'=y'_1\cdots y'_l,\ l\geq 1.
\end{equation}
Clearly, $$\theta(v,x_w,y_w,y')=\dim_B \big(\bigcup_{w'\in E^l: y_{w'}=y',\atop i(w')\in t(v,[w])} \phi_w\circ\phi_{w'}(\pi(X_{t(w')})) \big)=\max_{w'\in E^l:  y_{w'}=y' ,\atop i(w')\in t(v,[w]) }\dim_B\pi (X_{t(w')}).$$

For $k\geq 1$, define
\begin{align*}
\beta_k=\min_{w\in E^*:\atop |w|\geq \#V}\max_{v\in t([w])}\Big\{ \min \big \{    &\sum_{[w']\in [E^k]:\atop i(w')=v, y_{w'}=y}  n^{k\theta(v,x_{w'},y_{w'},y')}\cdot {1}_{\{\theta(v,x_{w'},y_{w'},y')>0\}}:
\\
&\sum_{[w']\in [E^k]:\atop i(w')=v, y_{w'}=y} \theta(v,x_{w'},y_{w'},y')>0, y\in \Sigma_Y^k,y'\in \Sigma_Y^* \big \}\Big\}.
\end{align*}
%Since the range of $\theta$ is finite and $t([w]) \subseteq V$, the  minimum is reasonable.
%It can be directly see that $\beta_k >0.$
Write
\begin{equation}\label{el}
 \beta=\liminf_{k\rightarrow \infty} (\beta_k)^{1/k}.
\end{equation}
In Section \ref{sec4}, we will prove that the lower dimension of $X$ is $\frac{\log \beta}{\log n}.$
\begin{remark}
    For  $v\in V$, $w\in E^*$, $i(w)=v$, we have
    $$
    \pi\left(X_v\cap \psi_w((0,1)^2)\right) \cap \phi_w(I_{y'}) \subseteq \pi\left(X_v\cap \psi_w([0,1]^2)\right),
    $$
    which gives  $\eta(v,x_w,y_w)=\theta(v,x_w,y_w,\emptyset)\geq \theta(v,x_w,y_w,y')$, for all $y'$ in $ \Sigma_Y^*$. In addition,
    $$
    \alpha_k=\max_{v\in \tilde{V}}\max_{y\in \Sigma_Y^k} \max_{y'\in \Sigma_Y^*}  \sum_{[w]\in [E^k]:\atop i(w)=v, y_{w}=y } n^{k\theta(v,x_{w},y_{w},y')}.
    $$
\end{remark}

\subsection{Approximate squares}

Before proceeding, let us look at the definition of Assouad dimension. Given a bounded set $F\subseteq \mathbb{R}^{d}$, the \textit{Assouad dimension} of $F$ is
     $$
     \begin{aligned}
    \dim_{A}F=\inf\Big \{s\geq0:\exists C>0, &\text{ for any } 0<r<R\leq 1 \text{ and }  x\in F,\\
    &\mathcal{N}_r\big (F\cap B(x,R)\big )\leq C\left(\frac{R}{r}\right)^s\Big \},
    \end{aligned}
    $$
    where $\mathcal{N}_r(E)$ is the least number of balls of radius $r$ covering the set $E$, and $B(x, R)$ is the ball with center  $x$ in $\mathbb{R}^d$ and radius $R>0$. See \cite{F21} for more details.
    \vspace{0.2cm}

An important concept in the dimension theory of self-affine sets is the so called `approximate square'. For $k\geq 1$, choose $l\leq k$ so that $n^l\leq m^k<n^{l+1}$, i.e. $l=\lfloor k\log_nm\rfloor$, the maximal integer less than or equal to $k \log_n m$. For $p,q$ in $\mathbb{Z}$, denote
$$
Q_k(p,q)=[\frac{p}{n^l},\frac{p+1}{n^l}]\times[\frac{q}{m^k},\frac{q+1}{m^k}]
$$
the \textit{approximate square} of level $k$ at $(p,q)$.

For a bounded set $E\subseteq \mathbb{R}^2$, for $k\geq 1$, let $N_{k}(E)$  be the least number of  elements in $\{Q_{k}(p,q):p,q\in \mathbb{Z}\}$ covering $E$. In a standard way, we may replace $\mathcal{N}_r(X\cap B(x,R))$ with  $N_{k'}(X\cap Q_{k}(p,q))$ in the definition of Assouad dimension of $X$, i.e.
$$
     \begin{aligned}
    \dim_{A}X =\inf\Big \{s\geq0:\exists C>0, \text{ for any } k'> k &\text{ and }  Q_k(p,q) \text{ with } X\cap Q_k(p,q) \neq \emptyset,\\
    &{N}_{k'}\big (X\cap Q_k(p,q)\big )\leq C m^{(k'-k)s}\Big \}.
    \end{aligned}
    $$

    Similar to the $2$-dimensional case, for a bounded set $E\subseteq \mathbb{R}$, $k\geq 1$, we write $N_{k}\big (E)$ the least number of intervals in the form $[\frac{i}{m^{k}},\frac{i+1}{m^{k}} ]$ ($i$  $\in \mathbb{Z}$) covering $E$.

\vspace{0.2cm}

    \begin{lemma}\label{le3}
        For any $\epsilon>0,$ there exists a constant $c_1>0$, such that for each $v\in V$, $w\in E^*$ with $i(w)=v$,   $k\geq 1$ and $y\in \Sigma_Y^* \setminus \{\emptyset\}$, we have
        \begin{equation}\label{e13}
        c_1^{-1}m^{k(\eta(v,x_w,y_w)-\epsilon)}\leq N_{k+|w|}\left(\pi(X_v \cap \psi_w((0,1)^2))\right) \leq c_1 m^{k(\eta(v,x_w,y_w)+\epsilon)}
        \end{equation}
        and
        \begin{equation}\label{e14}
        c_1^{-1}m^{k(\theta(v,x_w,y_w,y)-\epsilon)}\leq N_{k+|w|+|y|}\left(\pi(X_v \cap \psi_w((0,1)^2))\cap \phi_w(I_y)\right) \leq c_1 m^{k(\theta(v,x_w,y_w,y)+\epsilon)},
        \end{equation}
        for those
        $$
        I_y=\big (\sum_{i=1}^{|y|}\frac{y_i}{m^i},\sum_{i=1}^{|y|}\frac{y_i}{m^i}+\frac{1}{m^{|y|}}\big )
        $$
        satisfying $\pi(X_v \cap \psi_w((0,1)^2))\cap \phi_w(I_y) \neq \emptyset$.
    \end{lemma}
    \begin{proof}

    For $v\in V$, write $\lambda_v:=\dim_B\pi(X_v)$ for short. It follows immediately from the definition of box dimension that there exists a constant $C>0$, such that for $k\geq 1$ we have
    \begin{equation}\label{e5}
    C^{-1} m^{k(\lambda_v-\epsilon)}\leq N_k(\pi(X_v))\leq C m^{k(\lambda_v+\epsilon)}, \quad \text{ for all }v\in V.
    \end{equation}
    Note that
    $$
    \begin{aligned}
    N_{k+|w|}\left(\pi(X_v \cap \psi_w((0,1)^2)\right)=N_{k+|w|}\big({\phi_w(\bigcup_{v'\in t(v,[w])} \pi(X_{v'})}) \big)
    =N_k\big({\bigcup_{v'\in t(v,[w])} \pi(X_{v'})} \big).
    \end{aligned}
    $$
    Combining  \eqref{e17} and \eqref{e5}, we have
    $$
    \begin{aligned}
    N_{k+|w|}\left(\pi(X_v \cap \psi_w((0,1)^2)\right) &\leq \#V \max_{v'\in t(v,[w])} N_k(\pi(X_{v'}))\\
    &\leq \# V \max_{v'\in t(v,[w])}C m^{k(\lambda_{v'}+\epsilon)}\\
    &=\# V \cdot C m^{k(\eta(v,x_w,y_w)+\epsilon)},
    \end{aligned}
    $$
    and
    $$
    \begin{aligned}
    N_{k+|w|}\left(\pi(X_v \cap \psi_w((0,1)^2)\right) &\geq \max_{v'\in t(v,[w])} N_k(\pi(X_{v'}))\\
    &\geq  \max_{v'\in t(v,[w])}C^{-1} m^{k(\lambda_{v'}-\epsilon)}\\
    &=C^{-1} m^{k(\eta(v,x_w,y_w)-\epsilon)},
    \end{aligned}
    $$
    which gives \eqref{e13} by taking $c_1=\# V \cdot C$.
    \vspace{0.2cm}

    Noticing that
    $$
    \begin{aligned}
    N_{k+|w|+|y|}\left(\pi(X_v \cap \psi_w((0,1)^2)\cap \phi_w(I_y) \right)&=N_{k+|w|+|y|}\big({\phi_{w} (\bigcup_{w'\in E^l:y_{w'}=y,\atop i(w')\in t(v,[w])} \phi_{w'}\circ \pi(X_{t(w')})}) \big)
    \\
    &=N_{k}\big( \bigcup_{w'\in E^l:y_{w'}=y, \atop i(w')\in t(v,[w])}  \pi(X_{t(w')}) \big),
    \end{aligned}
    $$
    \eqref{e14} follows in a similar way.
    \end{proof}
\vspace{0.2cm}

\section{Assouad dimension}\label{sec3}
The main aim in this section is to  prove Theorem \ref{th1} for the Assouad dimension.

For $k\geq 1$, Write $$\tau_k=\max_{v\in V}\#\{[w]:w\in E^k,i(w)=v\}.$$
The following lemma is useful for the upper bound estimate of Assouad dimension.

\begin{lemma}\label{le1}
    The limit of $\{(\tau_k)^{1/k}\}_{k\geq 1}$ exists and $\tau:=\lim_{k\to \infty}(\tau_k)^{1/k} \leq \alpha$.
\end{lemma}
\begin{proof}
The existence of the limit of $\{(\tau_k)^{1/k}\}_{k\geq 1}$ follows in a similar way as the proof of Lemma \ref{lim}.

We prove $\tau \leq \alpha$ through following two steps: (i) $\frac{\log \tau}{\log n} \leq {\dim}_B X$, (ii) ${\dim}_B X\leq \frac{\log \alpha}{\log n}$ (the existence of $\dim_BX$ is ensured in \cite{F23}).

    \vspace{0.2cm}
    \textit{Step (i): $\frac{\log \tau}{\log n} \leq {\dim}_B X$.}
    \vspace{0.2cm}

    By the definition of $\tau$, for $\epsilon>0$, there exists $C>0$ such that for $l\geq 1$, we have
    $$
    \max_{v\in V}\#\{[w]:w\in E^l,i(w)=v\}=\tau_l\geq C(\tau-\epsilon)^l.
    $$
    For each $k\geq 1,$ choosing $l=\lfloor k\log_n m\rfloor$, we have
    $$
    N_{k}(X) \geq\max_{v\in V} N_{k}(X_v) \geq C (\tau-\epsilon)^l.
    $$
    So ${\dim}_B X\geq \frac{\log(\tau-\epsilon)}{\log n}$, thus
    ${\dim}_B X \geq \frac{\log \tau}{\log n}$ by the arbitrary of $\epsilon$.

    \vspace{0.2cm}
    \textit{Step (ii): ${\dim}_B X\leq \frac{\log \alpha}{\log n}$.}
    \vspace{0.2cm}

    Write  $\zeta= {\dim}_B X$. 
    %There exists a sequence $\{k_s\}_{s\geq 1}$ such that $\lim_{s\to \infty} \frac{\log N_{k_s}(X)}{k_s \log m}=\zeta$. 
    For each $k$, write $l(k)=\lfloor k\log_n m\rfloor$ and $l'(k)=\lfloor l(k) \log_n m \rfloor$. Note that for $\epsilon>0,$ there exists $C'>0,$ such that for any $k\geq 1$, we have
     \begin{equation}\label{e15}
     N_{l(k)}(X)\leq C'm^{l(k)(\zeta+\epsilon)}.
    \end{equation}

    Fix $k$ and choose a collection $\{Q_{l(k)}(p,q)\}$ that cover $X$ with $\#\{Q_{l(k)}(p,q)\}=N_{l(k)}(X)$. For each $Q_{l(k)}(p,q)$. We now turn to estimate $N_{k}(X\cap Q^\circ_{l(k)}(p,q))$, where $Q_{l(k)}^\circ(p,q)$ denotes the interior of $Q_{l(k)}(p,q)$.

\begin{figure}[htp]
\includegraphics[width=6.5cm]{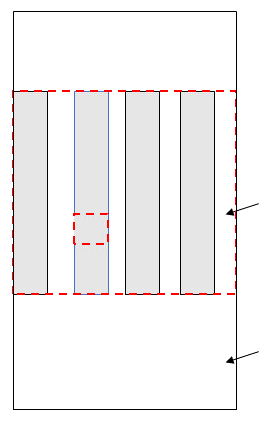}
  \begin{picture}(0,0)
\put(-5,52){$\scriptstyle{\psi_{w} ([0,1]^2)}$}
\put(-5,154){$\scriptstyle{Q_{l(k)}(p,q)}$}
\put(-148,174){$\scriptstyle{\psi_{\tilde{w}}([0,1]^2)}$}
\put(-131,137){$\scriptstyle{Q_{k}}$}
  \end{picture}
	\caption{The location of $Q_{l_s}(p,q)$ (resp. $Q_{k_s}$) in $\psi_w([0,1]^2)$ (resp. $\psi_{\tilde{w}}([0,1]^2)$).}
	\label{fig3.1}
\end{figure}
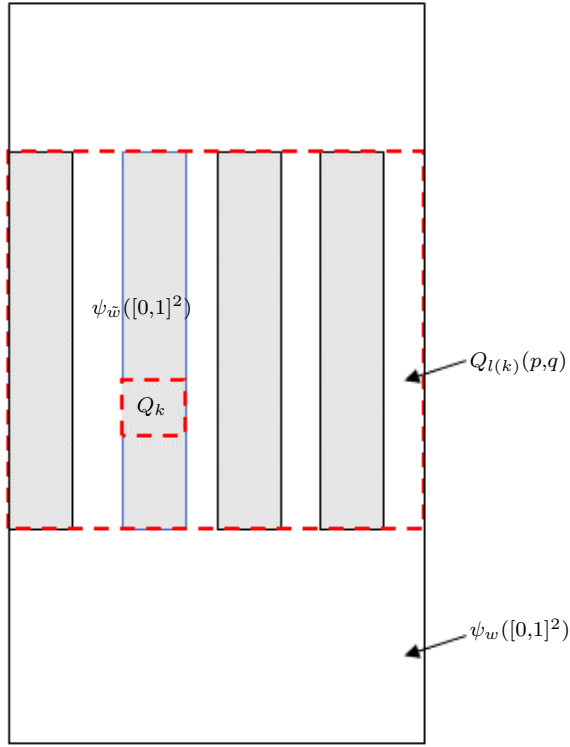

    For $v\in V,$ assume that  $X_v \cap Q^\circ_{l(k)}(p,q)\neq \emptyset$.  Then there exists $w\in E^{l'(k)}$ with $i(w)=v$ such that $Q^\circ_{l(k)}(p,q)\subseteq \psi_w((0,1)^2)$. Note that for $\tilde{w}\in E^{l(k)}$, $\psi_{\tilde{w}}([0,1]^2)$ is a rectangle with width $n^{-{l(k)}}$ and height $m^{-{l(k)}},$ and the height of $Q_{l(k)}(p,q)$ is $m^{-{l(k)}}$. By firstly dividing $Q_{l(k)}(p,q)$ into rectangles in terms of  $\psi_{\tilde{w}}([0,1]^2)$ with $\tilde{w}\in E^{l(k)}$, then covering each rectangle with  approximate squares of level $k$  (see Figure \ref{fig3.1} for an illustration), we have
    \begin{align}
    N_{k}(X_v\cap Q^\circ_{l(k)}(p,q))
    &\leq \sum_{v'\in t(v,[w])} \sum_{[w']\in [E^{l(k)-l'(k)}]:i(w')=v',\atop \psi_{w}\circ\psi_{w'}((0,1)^2) \subseteq Q^\circ_{l(k)}(p,q)} N_{k}\big (  \psi_w(X_{v'} \cap  \psi_{w'}((0,1)^2))\big )\nonumber
    \\
    &= \sum_{v'\in t(v,[w])} \sum_{[w']\in [E^{l(k)-l'(k)}]:i(w')=v',\atop \psi_{w}\circ\psi_{w'}((0,1)^2) \subseteq Q^\circ_{l(k)}(p,q)} N_{k}\big (  \phi_w \circ \pi (X_{v'} \cap  \psi_{w'}((0,1)^2))\big )\nonumber
    \\
    &=  \sum_{v'\in t(v,[w])} \sum_{[w']\in [E^{l(k)-l'(k)}]:i(w')=v',\atop \psi_{w}\circ\psi_{w'}((0,1)^2) \subseteq Q^\circ_{l(k)}(p,q)} N_{k-l'(k)}\big (   \pi (X_{v'} \cap  \psi_{w'}((0,1)^2))\big )\nonumber
    \\
    &\leq \sum_{v'\in t(v,[w])} \sum_{[w']\in [E^{l(k)-l'(k)}]:i(w')=v',\atop \psi_{w}\circ\psi_{w'}((0,1)^2) \subseteq Q^\circ_{l(k)}(p,q)} c_1 m^{(k-l(k))(\eta(v',x_{w'},y_{w'})+\epsilon)} \quad \text{by Lemma }\ref{le3}\nonumber
    \\
    &\leq c_1 n^{1+\epsilon} n^{(l(k)-l'(k))\epsilon}\cdot  \#V \max_{v'\in t(v,[w])} \sum_{[w']\in [E^{l(k)-l'(k)}]:i(w')=v',\atop \psi_{w}\circ\psi_{w'}((0,1)^2) \subseteq Q^\circ_{l(k)}(p,q)}n^{(l(k)-l'(k))\eta(v',x_{w'},y_{w'})}\nonumber
    \\
    &\leq c_1 n^{1+\epsilon} n^{(l(k)-l'(k))\epsilon}\cdot \#V \alpha_{l(k)-l'(k)}.\label{e19}
    \end{align}
    Combining this with  \eqref{e15}, we have
    $$
    \begin{aligned}
    N_{k}(X)&\leq C' m^{l(k)(\zeta+\epsilon)} \cdot (\#V)^2 c_1n^{1+\epsilon} n^{(l(k)-l'(k))\epsilon} \alpha_{l(k)-l'(k)},
    \end{aligned}
    $$
    which yields that
    $$
    {\dim}_B X \leq (\zeta+\epsilon)\log_n m+(\frac{\log \alpha}{\log n}+\epsilon) (1-\log_n m)
    $$ by using  $l(k)/k \to \log_n m$ as $k\to \infty$. Therefore, $\dim_B X \leq \frac{\log \alpha }{\log n }+ \frac{1}{1-\log_n m}\epsilon $. So $\dim_B X \leq \frac{\log \alpha }{\log n}$ by the arbitrary of $\epsilon.$
\end{proof}

Now we turn to estimate the upper bound of Assouad dimension of $X.$

%The upper bound of Assouad dimension reduces the following lemma.

\begin{lemma}\label{le4}
    For $\epsilon>0$, there exists a constant $C>0$ such that for $1\leq k\leq k'$ and any $p,q\in \mathbb{Z}$, we have
    \begin{equation}
    N_{k'}(X\cap Q^\circ_k(p,q))\leq C m^{(k'-k)(\frac{\log \alpha+\epsilon}{\log n}+\epsilon)}. \nonumber
    \end{equation}
\end{lemma}
\begin{proof}
    Let $l=\lfloor k\log_n m\rfloor $ and $l'=\lfloor k'\log_n m \rfloor$. It suffices to show that for each $v\in V$,  we have
    \begin{equation}\label{e7}
    N_{k'}(X_v\cap Q^\circ_k(p,q))\leq C n^{(l'-l)(\frac{\log \alpha+\epsilon}{\log n}+\epsilon)},
    \end{equation}
    for all $k'\geq k$ with  $k\geq \#V \frac{2}{\log _n m }$. Assume that $X_v\cap Q_k^{\circ}(p,q)\neq \emptyset$. There exists $w\in E^l$ such that $i(w)=v, Q^\circ_k(p,q)\subseteq \psi_{w}((0,1)^2)$. Recall that $\alpha=\lim_{s\to \infty} (\alpha_s)^{1/s},$ there exists $S$ such that for $s\geq S$,
    \begin{equation}
    \frac{\log \alpha_s}{s}\leq \log \alpha +\epsilon.\nonumber
    \end{equation}
    Therefore, there exists a constant $C_1>0$ such that for $s\geq 1$,
    \begin{equation}
        \label{e6}
        \alpha_s \leq C_1 n^{s(\frac{\log \alpha+\epsilon}{\log n})}.
    \end{equation}

    Consider the following two cases.

\begin{figure}[htp]
	\includegraphics[width=4.5cm]{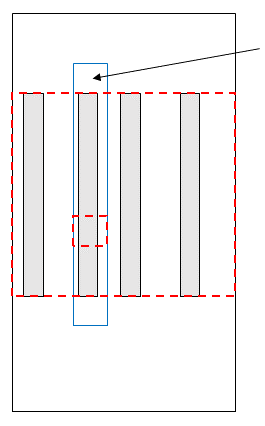}\hspace{2cm}
 \begin{picture}(0,0)
 \put(-135,-5){(i)}

\put(-70,20){$\scriptstyle{\psi_w(X_{v'})}$}
\put(-60,180){$\scriptstyle{\psi_{\tilde{w}}([0,1]^2)}$}
  \end{picture}
    \includegraphics[width=4.5cm]{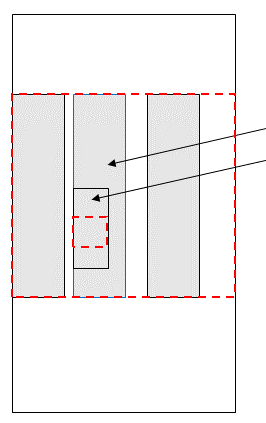}
    \begin{picture}(0,0)
    \put(-79,-7){(ii)}
    \put(0,143){$\scriptstyle{\psi_{\tilde{w}}([0,1]^2)}$}
\put(0,125){$\scriptstyle{\psi_{\tilde{w}}\circ\psi_{\hat{w}}([0,1]^2)}$}
\put(-10,20){$\scriptstyle{\psi_w(X_{v'})}$}
  \end{picture}
	\caption{Covering $Q^\circ_k(p,q)$ with approximate squares of level $k'$ in Cases (i), (ii).}
	\label{fig3.2}
\end{figure}

    \vspace{0.2cm}
    \textit{Case (i): $l'< k$.}
    \vspace{0.2cm}

    In this case, by firstly dividing $Q_{k}(p,q)$ into rectangles in terms of  $\psi_{\tilde{w}}([0,1]^2)\cap Q_k(p,q)$ with $\tilde{w}\in E^{l'}$, then covering each rectangle with  approximate squares of level $k'$ (see Figure \ref{fig3.2}-(i)), we have

    \begin{align}
        &N_{k'}(X_v\cap Q^\circ_k(p,q))\nonumber
        \\
         \leq& \sum_{v'\in t(v,[w])} \sum_{[w']\in [E^{l'-l}]: i(w')\in  v',\atop \psi_{w}\circ\psi_{w'} ((0,1)^2))\cap Q_k^\circ(p,q)\neq \emptyset} N_{k'} \big( \psi_w(X_{v'} \cap  \psi_{w'}((0,1)^2))\cap Q^\circ_{k}(p,q) \big) \nonumber
        \\
        =&\sum_{v'\in t(v,[w])} \sum_{[w']\in[E^{l'-l}]: i(w')=v', \atop \psi_{w}\circ\psi_{w'} ((0,1)^2))\cap Q_k^\circ(p,q)\neq \emptyset} N_{k'} \big( \phi_w\circ \pi (X_{v'} \cap  \psi_{w'}((0,1)^2)\cap  \psi_w^{-1}( Q^\circ_{k}(p,q))) \big) \nonumber
        \\
        =&\sum_{v'\in t(v,[w])} \sum_{[w']\in [E^{l'-l}]: i(w')=v',\atop \psi_{w}\circ\psi_{w'}((0,1)^2))\cap Q_k^\circ(p,q)\neq \emptyset} N_{k'-l} \big(  \pi (X_{v'} \cap  \psi_{w'}((0,1)^2)\cap  \psi_w^{-1}( Q^\circ_{k}(p,q))) \big) \label{e20}
        \\
        \leq& \sum_{v'\in t(v,[w])}  \sum_{[w']\in [E^{l'-l}]: i(w')=v',\atop \psi_{w}\circ\psi_{w'} ((0,1)^2))\cap Q_k^\circ(p,q)\neq \emptyset}  c_1 m^{(k'-k)(\theta(v',x_{w'},y_{w'},y)+\epsilon)} \quad \text{by Lemma }\ref{le3} \nonumber
        \\
        \leq& \sum_{v'\in t(v,[w])}    \sum_{[w']\in [E^{l'-l}]: i(w')=v',\atop \psi_{w}\circ\psi_{w'} ((0,1)^2))\cap Q_k^\circ(p,q)\neq \emptyset}  c_1 m^{(k'-k)(\eta(v',x_{w'},y_{w'})+\epsilon)} \nonumber
        \\
        \leq& c_1\#V\cdot  \alpha_{l'-l} n^{(l'-l+1)\epsilon+1} \quad \text{ since } l\geq k\log_n m-1 \geq \#V\nonumber
        \\
        \leq& c_1\#V\cdot n^{\epsilon+1} C_1 n^{(l'-l)(\frac{\log \alpha+\epsilon }{\log n}+\epsilon)} \quad \text{ by }\eqref{e6}\nonumber
    \end{align}
    where in the forth line $y$ is in $ \Sigma_Y^{k-l'}$ so that
    \begin{equation}\label{e16}
    \phi_{w'}^{-1}\circ \pi \circ \psi_w^{-1}(Q^\circ_k(p,q))=I_y, \quad \text{ recall }\eqref{e18},
    \end{equation}
    and the second to last line follows in a same way as \eqref{e19}.
    %for some $w'\in E^{l'-l}$ with $i(w')\in t(v,[w])$ and $\psi_{w'}((0,1)^2) \cap \psi_w^{-1}(Q_k^\circ(p,q))\neq \emptyset$ (such that  $y_{w'}$ are same).

   % Using \eqref{e6},  when $l'-l\geq N$, we have \eqref{e7}. Besides,  when $l'-l<N$, \eqref{e7} is still true by adjusting the constant $C.$

    \vspace{0.2cm}
    \textit{Case (ii): $k\leq l'$.}
    \vspace{0.2cm}

In this case, by firstly dividing $Q_{k}(p,q)$ into rectangles in terms of  $\psi_{\tilde{w}}([0,1]^2)$ with $\tilde{w}\in E^{k}$, secondly dividing  each rectangle into smaller ones in terms of  $\psi_{\tilde{w}}\circ\psi_{\hat{w}}([0,1]^2)$ with $\hat{w}\in E^{l'-k}$, then covering each smaller rectangle with approximate squares of level $k'$ (see Figure \ref{fig3.2}-(ii)), we have

    \begin{align*}
        &N_{k'}(X_v\cap Q^\circ_k(p,q))\\
        \leq&  \sum_{v'\in t(v,[w])} \sum_{[w']\in [E^{k-l}]:i(w')=v', \atop \psi_{w}\circ\psi_{w'}((0,1)^2)\subseteq Q^\circ_k(p,q)} N_{k'}\big ( \psi_w (X_{v'} \cap \psi_{w'}((0,1)^2)) \big)
        \\
        \leq&\sum_{v'\in t(v,[w])} \sum_{[w']\in [E^{k-l}]:i(w')=v', \atop \psi_{w}\circ\psi_{w'}((0,1)^2)\subseteq Q^\circ_k(p,q)} \sum_{v''\in t(v',[w'])} \sum_{[w'']\in [E^{l'-k}]:\atop i(w'')=v''}  N_{k'} \big ( \psi_w\circ \psi_{w'} (X_{v''} \cap  \psi_{w''}((0,1)^2)) \big)
        \\
        =&\sum_{v'\in t(v,[w])} \sum_{[w']\in [E^{k-l}]:i(w')=v', \atop \psi_{w}\circ\psi_{w'} ([0,1]^2)\subseteq Q^\circ_k(p,q)} \sum_{v''\in t(v',[w'])} \sum_{[w'']\in [E^{l'-k}]:\atop i(w'') =v''}   N_{k'} \big ( \phi_w\circ \phi_{w'}\circ \pi  (X_{v''} \cap  \psi_{w''}([0,1]^2)) \big)
        \\
        =&\sum_{v'\in t(v,[w])} \sum_{[w']\in [E^{k-l}]:i(w')=v', \atop \psi_{w}\circ\psi_{w'} ([0,1]^2)\subseteq Q^\circ_k(p,q)} \sum_{v''\in t(v',[w'])} \sum_{[w'']\in [E^{l'-k}]:\atop i(w'') =v''}   N_{k'-k} \big (  \pi  (X_{v''} \cap  \psi_{w''}([0,1]^2)) \big)
        \\
        \leq&\sum_{v'\in t(v,[w])} \sum_{[w']\in [E^{k-l}]:i(w')=v', \atop \psi_{w}\circ\psi_{w'}([0,1]^2)\subseteq Q^\circ_k(p,q)} \sum_{v''\in t(v',[w'])} \sum_{[w'']\in [E^{l'-k}]:\atop i(w'') =v''} c_1 m^{(k'-l')(\eta(v'',x_{w''},y_{w''})+\epsilon)} \quad \text{ by Lemma }\ref{le3}
        \\
        \leq & \sum_{v'\in t(v,[w])} \sum_{[w']\in [E^{k-l}]:i(w')=v', \atop \psi_{w}\circ\psi_{w'}([0,1]^2)\subseteq Q^\circ_k(p,q)} c_1 m^{(k'-l')(\eta(v',x_{w'},y_{w'}+\epsilon)} \cdot \#V \tau_{l'-k}
        \\
        \leq& c_1(\#V)^2\cdot \alpha_{k-l} n^{(k-l+1)\epsilon+1} \cdot \tau_{l'-k} \quad \text{ since } l=\lfloor k\log_n m\rfloor \geq \#V.
    \end{align*}
    %where we always require that $\psi_w(X_{t(w)})\subseteq Q_k(p,q)$ and $\psi_{ww'}(X_{t(w')})\subseteq Q_k(p,q)$.
    On the other hand, by Lemma \ref{le1}, $\log \tau \leq\log  \alpha$, so there exists $C_2>0$ such that for all $s\geq 1$,
    \begin{equation}\label{e8}
       \tau_s  \leq C_2n^{s(
       \frac{\log \alpha+\epsilon}{\log n })}.
    \end{equation}
    Therefore, combining the above estimate with  \eqref{e6} and \eqref{e8}, we have
    $$
    \begin{aligned}
    N_{k'}(X_v\cap Q^\circ_k(p,q))&\leq c_1 n^{\epsilon+1}(\#V)^2 \cdot C_1n^{(k-l)(\frac{\log \alpha +\epsilon}{\log n}+\epsilon)}\cdot C_2 n^{(l'-k)(\frac{\log \alpha+\epsilon}{\log n})}
    \\
    &\leq c_1 n^{\epsilon+1}(\#V)^2\cdot C_1 C_2 \cdot n^{(l'-l)(\frac{\log \alpha+\epsilon}{\log n }+\epsilon)},
    \end{aligned}
    $$
    which gives \eqref{e7}.
\end{proof}

Next we estimate the  Assouad dimension of $X$ from below.
\begin{lemma}\label{le2}
    For any small  $\epsilon>0$, for any $C>0$, there exist $k'\geq k\geq 1$, $p,q$ in $ \mathbb{Z}$, such that
    $$N_{k'}\left(X\cap Q^\circ_k(p,q)\right)\geq Cm^{(k'-k)(\frac{\log \alpha-\epsilon}{\log n}-2\epsilon)}.
    $$
\end{lemma}
\begin{proof}
    By the definition of $\alpha$, there exists $C_1>0$ such that for $s \geq 1$, we have
    \begin{equation}\label{e10}
         \alpha_s \geq C_1 n^{s(\frac{ \log \alpha -\epsilon}{\log n })}.
    \end{equation}
    Fix $s$, let $v\in \tilde{V}$,  $y=y_1\cdots y_s\in \Sigma_Y^s$  such that
    $$
    \alpha_s=\sum_{[w]\in [E^s]:\atop i(w)=v, y_w =y } n^{s\eta(v,x_w,y_w)}.
    $$

\begin{figure}[htp]
\includegraphics[width=6.5cm]{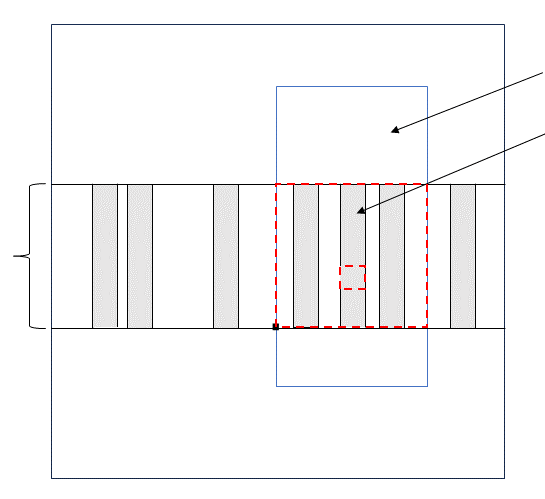}
  \begin{picture}(0,0)
\put(0,124){$\scriptstyle{\psi_{\tilde{w}}\circ\psi_{w'} ([0,1]^2)}$}
\put(0,146){$\scriptstyle{\psi_{\tilde{w}}(X_v)}$}
\put(-215,80){$\scriptstyle{\phi_{\tilde{w}}(I_y)}$}
\put(-115,50){$\scriptstyle{(p,q)}$}
\put(-10,15){$\scriptstyle{X_{\tilde{v}}}$}
  \end{picture}
	\caption{The choice of $Q_k(p,q)$ according to $s,v, \tilde{v}$ and $y$.}
	\label{fig3.3}
\end{figure}
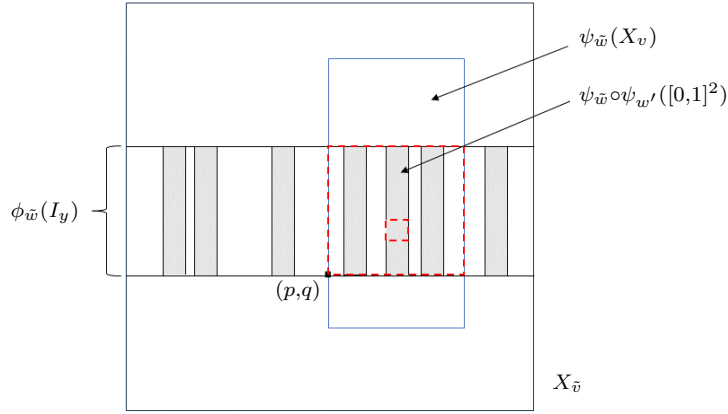

    Choose a $k\geq 1$ such that $k=\lfloor k\log_n m\rfloor+s$ and let $l=\lfloor k\log_n m \rfloor$. Since $v\in \tilde{V}$, there exists ${w}\in E^*$ with $|w|\geq \#V$ such that $t({w})=v$. Noticing that $\{i(w_1),\cdots, i(w_{|w|}), t(w)\} \subseteq V$, there must exist $\tilde{w}\in E^l$ such that $t(\tilde{w})=v$. Write $\tilde{v}=i(\tilde{w}).$ Let $p=\sum_{i=1}^l x_{\tilde{w}_i} n^{l-i}$ and $q=\sum_{i=1}^{l} y_{\tilde{w}_i} m^{k-i} +\sum_{i=1}^s y_i m^{s-i}$, therefore $X_{\tilde{v}}\cap Q^\circ_{k}(p,q)\neq \emptyset.$
    We also choose a $k'\geq 1$ such that $k=\lfloor k' \log_n m \rfloor$. Since $X_{\tilde{v}}\subseteq X$, it suffices to estimate $N_{k'}(X_{\tilde{v}}\cap Q^\circ_{k}(p,q)).$ Similar to the argument in Lemma \ref{le1} (see Figure \ref{fig3.3} for an illustration), we have
    \begin{align*}
        N_{k'}(X_{\tilde{v}}\cap Q^\circ_{k}(p,q))&\geq N_{k'}(\psi_{\tilde{w}}(X_{v})\cap Q^\circ_k(p,q))
        \\
        &=\sum_{[w']\in [E^{s}] :\atop i(w')=v, y_{w'}=y} N_{k'}\big (\psi_{\tilde{w}}(X_{v}) \cap \psi_{\tilde{w}}\circ\psi_{w'}((0,1)^2)\big )
        \\
        &=\sum_{[w']\in [E^{s}]:\atop i(w')=v, y_{w'}=y} N_{k'-l}\big (\pi(X_{v}\cap \psi_{w'}((0,1)^2))\big )
        \\
        &\geq \sum_{[w']\in [E^{s}]:\atop i(w')=v,y_{w'}=y} c_1^{-1} m^{(k'-k)(\eta(v,x_{w'},y_{w'})-\epsilon)} \quad \text{by Lemma }\ref{le3}
        \\
        &\geq c_1^{-1} n^{\epsilon-1} n^{-s\epsilon} \alpha_s
        \geq c_1^{-1} n^{\epsilon-1}C_1 n^{s(\frac{\log \alpha-\epsilon}{\log n}-\epsilon)} \quad \text{ by }\eqref{e10}.
    \end{align*}

   For any  $C>0$, there exists large enough $s\geq 1$ such that $c_1^{-1} n^{\epsilon-1}C_1 n^{-(\frac{\log \alpha-\epsilon}{\log n }-2\epsilon)} n^{s\epsilon}> C$. Therefore,
    $$
    N_{k'}(X\cap Q^\circ_{k}(p,q)) \geq c_1^{-1} n^{\epsilon-1}C_1 n^{s\epsilon} m^{(k'-k-\log_m n)(\frac{\log \alpha-\epsilon}{\log n}-2\epsilon)} > C m^{(k'-k)(\frac{\log \alpha-\epsilon}{\log n}-2\epsilon)},
    $$
    where $k,k',p,q$ are chosen as before according to $s$.
\end{proof}

\begin{proof}[Proof of Theorem \ref{th1} for Assouad dimension]
    Combining Lemmas \ref{le4} and \ref{le2}, letting $\epsilon$ go to $0$, we finally obtain $\dim_A X=\frac{\log \alpha}{\log n}$.
\end{proof}

\begin{proof}[Proof of Corollary \ref{c1}]
    By assumption, $\dim_B \pi(X)=\dim_B \pi(X_v)$ for all  $v$ in $V,$ and so
    $$
    \eta(v,x_w,y_w)=\dim_B \pi(X),
    $$
    for all $w\in E^*$ with $i(w)=v$.
    Thus
    $$\alpha_k=\max_{v\in V} \max_{y\in \Sigma_Y^k} n^{k\dim_B \pi(X)} \sum_{v'\in V} A_{y_1}\cdots A_{y_k}(v,v').$$
    Therefore,
    $$
    \log \alpha= \dim_B \pi(X) \log n + \lim_{k\to \infty}  \frac{1}{k}\log \max_{y\in \Sigma_Y^k} \Vert A_{y_1} \cdots A_{y_k} \Vert.
    $$
    By Theorem \ref{th1}, the expression of Assouad dimension follows.
\end{proof}

%\begin{proof}[Proof of Corollary \ref{th2}]
%On account of all $\dim_B \pi(X_v)$ are equal, and $$\beta=\dim_B \pi(X) =\max_{v\in V} \dim_B \pi(X_v)=\dim_B \pi(X_v), \quad \text{ for all }  v \in V .$$
%Recall that
%$$
%\alpha_k=\max_{j_1,\cdots,j_k\in \{0,\cdots, m-1\}} \sum_{v,v'\in V} n^{k\beta_{v'}}\# \{w\in E^k:v\stackrel{w}{\to}v', y_{w_1}=j_1,\cdots, y_{w_k}=j_k\}.
%$$
%Thus,
%$$
%\alpha_k = n^{k\beta} \max_{j_1,\cdots,j_k\in \{0,\cdots, m-1\}} \Vert A_{j_1} \cdots A_{j_k} \Vert_1,
%$$
%where $\Vert \cdot \Vert_1$ denotes the $l_1-$ norm.
%Therefore,
%$$
%\alpha =\lim_{k\to \infty} (\alpha_k)^{1/k} =n^{\beta} \cdot\lim_{k\to \infty} \left(\max_{j_1,\cdots,j_k\in \{0,\cdots, m-1\}} \Vert A_{j_1} \cdots A_{j_k} \Vert_1 \right)^{1/k},
%$$
%where the existence of last limit because the sequence is  submultiplictive. For any matrix norm $\Vert \cdot \Vert$, there exists $C>0$ such $C^{-1}\Vert \cdot \Vert_1 \leq \Vert \cdot \Vert \leq C \Vert \cdot \Vert_1$. Thus
%$$
%\alpha =n^{\beta} \cdot\lim_{k\to \infty} %\left(\max_{j_1,\cdots,j_k\in \{0,\cdots, m-1\}} \Vert A_{j_1} \cdots A_{j_k} \Vert \right)^{1/k},
%$$
%which gives \eqref{e13} by Theorem \ref{th1}. When $(V,E)$ is strongly connected, it is easy to see \eqref{e13} by Proposition \ref{p1}.
%\end{proof}

\section{Lower dimension}\label{sec4}
In this section, we look at the lower dimension $\dim_L X$ of $X$, which equals
 $$
 \begin{aligned}
    \dim_{L}X=\sup\Big \{s\geq0:\exists C>0, \text{ for any } k'>k &\text{ and } Q_k(p,q) \text{ with } X\cap Q_k(p,q)\neq \emptyset,\\
    &N_{k'}(X\cap Q_k(p,q))\geq Cm^{(k'-k)s} \Big \}.
    \end{aligned}
    $$

The following lemma aims to the lower bound of lower dimension.

\begin{lemma}\label{low1}
    For any small $\epsilon>0$, there exists a constant $C>0$ such that for $ k'\geq k\geq 1$ and any $Q_k(p,q)$ with $X\cap Q_k^\circ(p,q)\neq \emptyset$, we have
    \begin{equation}
    N_{k'}(X\cap Q^\circ_k(p,q))\geq C m^{(k'-k)(\frac{\log \beta-\epsilon}{\log n}-\epsilon)}. \nonumber
    \end{equation}
\end{lemma}

\begin{proof}

  Let $l=\lfloor k\log_n m\rfloor$ and $l'=\lfloor k' \log_n m\rfloor$. It suffices to show that
    \begin{equation}%\label{el4}
    N_{k'}(X\cap Q^\circ_k(p,q))\geq C n^{(l'-l)(\frac{\log \beta-\epsilon}{\log n}-\epsilon)},\nonumber
    \end{equation}
    for all $k'\geq k$ with $k\geq \#V \frac{2}{\log_n m}$.
    Since $X\cap Q^\circ_k(p,q)\neq \emptyset$, there exists $w\in E^l$ such that  $Q^\circ_k(p,q) \subseteq \psi_w((0,1)^2)$. Note that  $|w|\geq \#V$. According to \eqref{el}, there exists $C_1>0$ such that for $s\geq 1$,
    \begin{equation}\label{el5}
    \beta_s \geq C_1 n^{s(\frac{\log \beta-\epsilon}{\log n})}.
    \end{equation}

    \begin{figure}[htp]
\centering
	\includegraphics[width=4.35cm]{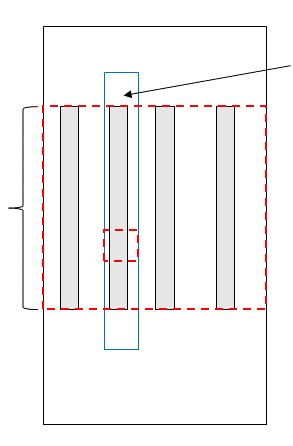}\hspace{2cm}
 \begin{picture}(0,0)
 \put(-123,-5){(i)}
 \put(-211,97){$\scriptstyle{\phi_w(I_y)}$}
\put(-65,20){$\scriptstyle{\psi_w(X_{v})}$}
\put(-55,160){$\scriptstyle{\psi_{w}\circ\psi_{w'}([0,1]^2)}$}
  \end{picture}
    \includegraphics[width=4.5cm]{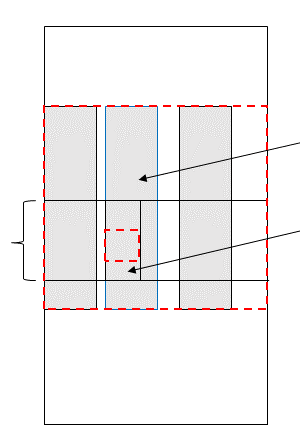}
    \begin{picture}(0,0)
    \put(-70,-5){(ii)}
    \put(0,127){$\scriptstyle{\psi_{w}([0,1]^2)}$}
\put(0,89){$\scriptstyle{\psi_{w}\circ\psi_{w'}([0,1]^2)}$}
\put(-10,20){$\scriptstyle{\psi_w(X_{v})}$}
\put(-164,83){$\scriptstyle{\phi_w(I_{y_{w'}})}$}
  \end{picture}
	\caption{Covering $Q^\circ_k(p,q)$ with approximate squares of level $k'$.}
	\label{fig4.1}
\end{figure}

 \vspace{0.2cm}
    \textit{Case (i): $l'<k$.}
    \vspace{0.2cm}

    This case is similar to  Case (i) in Lemma \ref{le4}. By \eqref{e20}, we have (see Figure \ref{fig4.1}-(i))
    \begin{align*}
        N_{k'}(X\cap Q^\circ_k(p,q))& \geq \max_{v\in t([w])}  N_{k'}(\psi_w(X_v)\cap Q^\circ_k(p,q))
        \\
        &\geq \max_{v\in t([w])} \sum_{[w']\in [E^{l'-l}]: i(w')=v,\atop \psi_{w}\circ\psi_{w'} ((0,1)^2))\cap Q_k^\circ(p,q)\neq \emptyset} N_{k'-l} \big(  \pi (X_{v} \cap  \psi_{w'}((0,1)^2)\cap  \psi_w^{-1}( Q^\circ_{k}(p,q))) \big)
        \\
        &\geq \max_{v\in t([w])} \sum_{[w']\in [E^{l'-l}]: i(w')=v,\atop \psi_{w}\circ\psi_{w'} ((0,1)^2))\cap Q_k^\circ(p,q)\neq \emptyset}  c_1^{-1} m^{(k'-k)(\theta(v,x_{w'},y_{w'},y)-\epsilon)}\cdot 1_{\{\theta(v,x_{w'},y_{w'},y)>0\}}
        %\\
        %&\geq \max_{v\in V} \max_{v'\in t(v,[w])} \sum_{[w']\in [E^{l'-l}]:i(w')=v',\atop \psi_w \circ \psi_{w'}([0,1]^2) \cap Q^\circ _k (p,q)\neq \emptyset} N_{k'}\big (\psi_w(X_{v'} \cap \psi_{w'}([0,1]^2))\cap Q_k(p,q)\big )
        %\\
        %&= \max_{v'\in t([w])} \sum_{[w']\in [E^{l'-l}]:i(w')=v',\atop \psi_w \circ \psi_{w'}([0,1]^2) \cap Q^\circ _k (p,q)\neq \emptyset} N_{k'}\big (\psi_w\circ \pi(X_{v'} \cap \psi_{w'}([0,1]^2)\cap \psi_w^{-1}( Q_k(p,q))\big )
        %\\
        %&= \max_{v'\in t([w])} \sum_{[w']\in [E^{l'-l}]:i(w')=v',\atop \psi_w \circ \psi_{w'}([0,1]^2) \cap Q^\circ _k (p,q)\neq \emptyset} N_{k'-l}\big ( \pi(X_{v'} \cap \psi_{w'}([0,1]^2)\cap \psi_w^{-1}( Q_k(p,q))\big )
        %\\
        %&\geq \max_{v'\in t([w])} \sum_{[w']\in [E^{l'-l}]:i(w')=v',\atop \psi_w \circ \psi_{w'}([0,1]^2) \cap Q^\circ _k (p,q)\neq \emptyset} c_2^{-1} m^{(k'-k)(\theta(v',x_{w'},y_{w'},y)-\epsilon)} \quad \text{by Lemma }\ref{le3}
        \\
        &\geq c_1^{-1} n^{\epsilon-1} n^{-(l'-l)\epsilon} \beta_{l'-l}
        \geq c_1^{-1} n^{\epsilon-1} \cdot C_1 n^{(l'-l)(\frac{\log \beta-\epsilon}{\log n }-\epsilon)} \quad \text{by }\eqref{el5}
    \end{align*}
    where $y$ is in $\Sigma_Y^{k-l'}$ satisfies \eqref{e16} and the third line follows from Lemma \ref{le3}.

    \vspace{0.2cm}

    \textit{Case (ii): $k\leq l'$.}
    \vspace{0.2cm}

    Write $p=\sum_{i=1}^l p_in^{l-i}$ and $q=\sum_{i=1}^k q_i{m^{k-i}}$ where $p_i$ (resp. $q_i$) is in $ \Sigma_X$ (resp. $\Sigma_Y$). So that $q_1\cdots q_l =y_w$. Denote $\tilde{y}=q_{l+1}\cdots q_k.$ Note that
    \begin{align*}
        N_{k'}(X\cap Q^\circ_k(p,q))&=  \sum_{y\in \Sigma_Y^{l'-k}} \sum_{[w']\in [E^{l'-l}]:\atop i(w')\in t([w]),y_{w'}=\tilde{y}y}N_{k'}(X\cap \psi_{w}\circ\psi_{w'}((0,1)^2)).
    \end{align*}
    Write $\kappa(y)=\sum_{[w']\in [E^{l'-l}]:\atop i(w')\in t([w]),y_{w'}=\tilde{y}y}N_{k'}(X\cap \psi_{w}\circ\psi_{w'}((0,1)^2))$, then  (see Figure \ref{fig4.1}-(ii))
    \begin{align*}
        N_{k'}(X\cap Q^\circ_k(p,q)) &\geq \#\{y\in \Sigma_Y^{l'-k}:\kappa(y)\neq 0\} \cdot \kappa(z)
        \\
        &=N_{l'}\big( \pi\circ \psi_w (\bigcup_{v\in t([w])} X_{v}) \cap \pi (Q^\circ_k(p,q))\big) \cdot \kappa(z)
        \\
        &\geq \max_{v\in t([w])} N_{l'-l} \big( \pi(X_{v}) \cap \pi\circ \psi_w^{-1}(Q^\circ_k(p,q))) \big )  \cdot \kappa(z)
        \\
        &\geq \max_{v\in t([w])} c_1^{-1} m^{(l'-k)(\theta(v,\emptyset,\emptyset,\tilde{y})-\epsilon)}\cdot \sum_{[w']\in [E^{l'-l}]:\atop i(w')=v,y_{w'}=\tilde{y}z}N_{k'}(\psi_w(X_v)\cap \psi_{w}\circ\psi_{w'}((0,1)^2))
        \\
        &\geq \max_{v\in t([w])} c_1^{-1} m^{(l'-k)(\theta(v,\emptyset,\emptyset,\tilde{y})-\epsilon)}\cdot \sum_{[w']\in [E^{l'-l}]:\atop i(w')=v, y_{w'}=\tilde{y}z} c_1^{-1}m^{(k'-l')(\eta(v,x_{w'},y_{w'})-\epsilon)}
        \\
        &\geq \max_{v\in t([w])} c_1^{-2} \sum_{[w']\in [E^{l'-l}]:\atop i(w')=v, y_{w'}=\tilde{y}z} m^{(k'-k)(\theta(v,x_{w'},y_{w'},\emptyset)-\epsilon)}
        \\
        &\geq c_1^{-2} n^{\epsilon-1} n^{-(l'-l)\epsilon} \beta_{l'-l}\geq c_1^{-2} n^{\epsilon-1} \cdot C_1 n^{(l'-l)(\frac{\log \beta-\epsilon}{\log n}-\epsilon)} \quad \text{ by }\eqref{el5}
    \end{align*}
    for some $z$ in $ \Sigma_Y^{l'-k}$ with $\kappa(z)=\min \{\kappa(y): \kappa(y)\neq 0, y\in \Sigma_Y^{l'-k}\}, $  where the forth and fifth lines are both from Lemma \ref{le3}.
    %Using \eqref{el5}, we have
    %$$
    %N_{k'}(X_v\cap Q_k(p,q)) \geq c_2^{-1} n^{\epsilon-1} \frac{1}{\#V} n^{(l'-l)(\frac{\log \beta-\epsilon}{\log n}-\epsilon)}
    %$$
    %which yields \eqref{el4}.

\end{proof}

Finally, we turn to the upper bound of the
lower dimension.

\begin{lemma}\label{low2}
    For any  $\epsilon>0$, for any $C>0$, there exist $k'\geq k\geq 1$, $p,q$ in $\mathbb{Z}$, such that
    $$N_{k'}\left(X\cap Q^\circ_k(p,q)\right)\leq Cm^{(k'-k)(\frac{\log \beta+\epsilon}{\log n}+2\epsilon)}.
    $$
\end{lemma}

\begin{proof}
It follows from \eqref{el}, there exist $C_1>0$ and a sequence $\{s_j\}_{j\geq 1}$ such that
\begin{equation}\label{el6}
    \beta_{s_j} \leq C_1 n^{s_j (\frac{\log \beta+\epsilon}{\log n})}.
\end{equation}
Fix large $j$, let $w\in E^*$ with $|w|\geq \#V$, $y=y_1\cdots y_{s_j}\in \Sigma_Y^{s_j}$ and $y'=y'_1 \cdots y'_h \in \Sigma_Y^h$   such that
$$
 \beta_{s_j}= \max_{v\in t([w])}\sum_{[w']\in [E^{s_j}]:\atop i(w')=v, y_{w'}=y}  n^{s_j\theta(v,x_{w'},y_{w'},y')}\cdot 1_{\{\theta(v,x_{w'},y_{w'},y')>0\}}.
$$

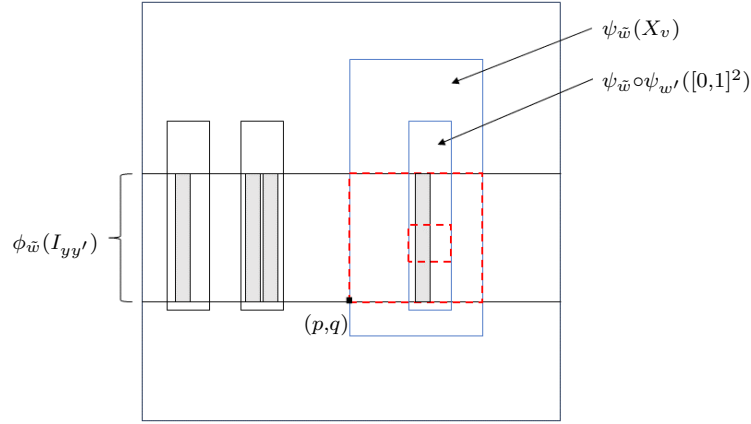
\begin{figure}[htp]
\centering
\includegraphics[width=6.5cm]{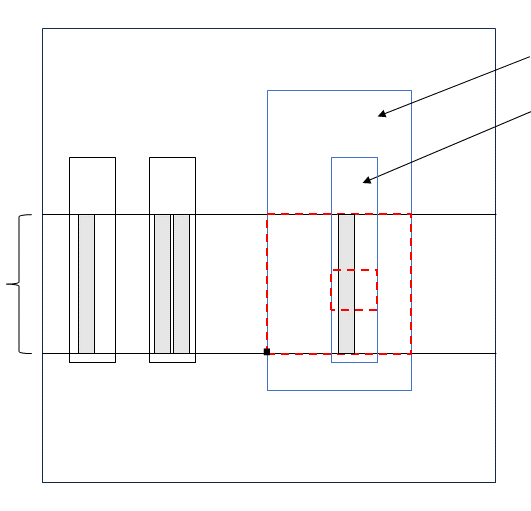}
  \begin{picture}(0,0)
\put(0,136){$\scriptstyle{\psi_{\tilde{w}}\circ\psi_{w'} ([0,1]^2)}$}
\put(0,156){$\scriptstyle{\psi_{\tilde{w}}(X_v)}$}
\put(-223,75){$\scriptstyle{\phi_{\tilde{w}}(I_{yy'})}$}
\put(-113,44){$\scriptstyle{(p,q)}$}
  \end{picture}
	\caption{The choice of $Q_k(p,q)$ according to $j$, $y$ and $y'$.}
	\label{fig4.2}
\end{figure}

Choose a $k\geq 1$ such that $k=\lfloor k\log_n m\rfloor+s_j+h$ and let $l=\lfloor k\log_n m \rfloor$. Since $|w|\geq \#V$ and $\{i(w_1),\cdots, i(w_{|w|}),t(w)\} \subseteq V$, there must  exist $\tilde{w}\in E^l$ such that $t(\tilde{w})\in t([w])$. Let $p=\sum_{i=1}^l x_{\tilde{w}_i} n^{l-i}$ and $q=\sum_{i=1}^{l} y_{\tilde{w}_i} m^{k-i} +\sum_{i=1}^{s_j} y_i m^{h+s_j-i}+\sum_{i=1}^h y'_i m^{h-i}$. Clearly $X\cap Q^\circ_k(p,q)\neq \emptyset$. Choose a $k'\geq 1$ such that $\lfloor k' \log_n m \rfloor=l+s_j$. Denote $l'=\lfloor k'\log_n m\rfloor.$ Similar to  Case (i) in Lemma \ref{le4}, by \eqref{e20}, we have (see Figure \ref{fig4.2})
\begin{align*}
        N_{k'}(X\cap Q^\circ_{k}(p,q))&\leq \#V \cdot \max_{v\in t([\tilde{w}])} N_{k'}(\psi_{\tilde{w}}(X_v) \cap Q^\circ_k(p,q))
        \\
        &\leq\#V \cdot \max_{v\in t([\tilde{w}])}  \sum_{[w']\in [E^{l'-l}]:i(w')=v,\atop \psi_{\tilde{w}}\circ\psi_{w'} ((0,1)^2) \cap Q^\circ _k (p,q)\neq \emptyset} N_{k'}\big (\psi_{\tilde{w}}(X_{v} \cap \psi_{w'}((0,1)^2))\cap Q^\circ_k(p,q)\big )
        \\
        &\leq\#V \cdot   \max_{v\in t([\tilde{w}])} \sum_{[w']\in [E^{l'-l}]:i(w')=v,\atop \psi_{\tilde{w}}\circ\psi_{w'} ((0,1)^2) \cap Q^\circ _k (p,q)\neq \emptyset} c_1 m^{(k'-k)(\theta(v,x_{w'},y_{w'},y')+\epsilon)} \cdot 1_{\{\theta(v,x_{w'},y_{w'},y')>0\}}
        \\
        &\leq\#V \cdot c_1 n^{\epsilon+1} n^{s_j\epsilon} \cdot  \beta_{s_j}
        \leq \#V \cdot c_1 C_1 n^{\epsilon+1}\cdot   n^{s_j(\frac{\log \beta+\epsilon}{\log n }+\epsilon)} \quad \text{by }\eqref{el6}.
    \end{align*}

   For any $C>0$, there exists $j\geq 1$ such that $\#V \cdot c_1 C_1  n^{\epsilon+1}  n^{\frac{\log \beta+\epsilon}{\log n}+2\epsilon}n^{-s_j \epsilon}< C$, therefore,
   \begin{equation}
    N_{k'}(X\cap Q^\circ_{k}(p,q))\leq \#V \cdot  c_1 C_1 n^{\epsilon+1}  n^{(\frac{\log \beta+\epsilon}{\log n}+2\epsilon)}n^{-s_j\epsilon} m^{(k'-k)(\frac{\log \beta+\epsilon}{\log n}+2\epsilon))}
    < C m^{(k'-k)(\frac{\log \beta+\epsilon}{\log n}+2\epsilon)}, \nonumber
   \end{equation}
   where $k,k',p,q$ are chosen as before according to $j$.
\end{proof}

\begin{proof}[Proof of Theorem \ref{th1} for lower dimension]
    Combining Lemmas \ref{low1} and \ref{low2}, we obtain that $\dim_L X=\frac{\log \beta}{\log n}$.
\end{proof}

\section{Comparison between box and Assouad dimensions}\label{sec5}
In this section, we prove Theorems \ref{th3} and \ref{th4}.   Recall that $$\mathcal{I}=\{(x_{\omega_1},y_{\omega_1})(x_{\omega_2},y_{\omega_2})\cdots:\omega \in E^\infty \} \subseteq \big \{ \{0,\cdots, n-1\}\times \{0,\cdots, m-1\} \big \}^{\mathbb{N}}
    $$
    and
    $$
    \pi \mathcal{I}=\{y_{\omega_1} y_{\omega_2}\cdots:\omega\in E^\infty\} \subseteq \{0,\cdots, m-1\}^{\mathbb{N}}.
    $$
 For $k$ in $\mathbb N$,  denote $\mathcal{I}^k$ (resp. $\pi\mathcal{I}^k$) the collection of words with length $k$ which appear in $\mathcal{I}$ (resp. $\pi\mathcal{I}$), that is
    $$
    \mathcal{I}^k=\{(x_{w_1},y_{w_1})(x_{w_2},y_{w_2})\cdots(x_{w_k},y_{w_k}):w \in E^k \}.
    $$
   {Recall that the \textit{topological entropy} of $\mathcal{I}$, $\pi \mathcal{I}$ and $\pi^{-1}(y)$ (for $y\in \pi \mathcal{I}$) are defined as $$\begin{aligned}
    h_{\textup{top}}(\mathcal{I}):&=\lim_{k\to \infty} \frac{1}{k}\log \# \mathcal{I}^k,\\
    h_{\textup{top}}(\pi\mathcal{I}):&=\lim_{k\to \infty} \frac{1}{k}\log \#\pi\mathcal{I}^k,\\
    h_{\textup{top}}(\pi^{-1}(y)):&=\limsup_{k\to \infty} \frac{1}{k} \log \#\{x_{w}:w\in E^k,y_w=y|_k\}
    \end{aligned}
    $$ (the existence of the first two limits follow by a submultiplicativity argument).  The following inequality is due to Bowen \cite{B71}:
    \begin{equation}\label{e22}
    h_{\textup{top}}(\mathcal{I}) \leq h_{\textup{top}}(\pi \mathcal{I}) +\sup_{y\in \pi \mathcal{I}} h_{\textup{top}}(\pi^{-1}(y)).
    \end{equation}
    \vspace{0.2cm}

Firstly, let us look at the irreducible case.
\begin{theorem}\label{th2}
    Let $X$ be same as in Theorem \ref{th1}. Assume that  $G=(V,E)$ is irreducible. We have
    \begin{equation}
    \dim_A X=\frac{h_{\textup{top}}(\pi \mathcal{I})}{\log m}+\sup_{y\in \pi \mathcal{I}}\frac{h_{\textup{top}}(\pi^{-1}(y))}{\log n}.\nonumber
    \end{equation}
    And the following three statements are equivalent:
    \begin{enumerate}[(i).]
        \item $\dim_B X =\dim_A X$,
        \item $\dim_H X =\dim_B X$,
        \item \eqref{e22} becomes an equality, i.e.
    \begin{equation}\label{e23}
    h_{\textup{top}}(\mathcal{I})=h_{\textup{top}}(\pi\mathcal{I})+\sup_{y\in \pi \mathcal{I}} h_{\textup{top}}(\pi^{-1}(y)).
    \end{equation}
    \end{enumerate}

\end{theorem}

Before proving this theorem, we prepare some lemmas. The following  lemma is a graph-directed version of the expression of $\dim_A X$ of Mackay \cite{M11}.

\begin{lemma}\label{le5}
    When $G$ is irreducible, we have
    \begin{equation}
    \dim_A X=\frac{h_{\textup{top}}(\pi \mathcal{I})}{\log m}+\sup_{y\in \pi \mathcal{I}}\frac{h_{\textup{top}}(\pi^{-1}(y))}{\log n}.\nonumber
    \end{equation}
\end{lemma}
\begin{proof}
    For $k\geq 1$, $y\in \Sigma_Y^k$ and $v,v'\in V$, denote $\mathcal{I}^k_{v,v'}(y)=\{(x_{w_1},y_{w_1})\cdots(x_{w_k},y_{w_k}):w\in E^k, v\stackrel{w}{\to}v',y_w=y\}$, $\mathcal{I}^k_{v}(y)=\bigcup_{v'\in V} \mathcal{I}^k_{v,v'}(y)$ and $\mathcal{I}^k(y)=\bigcup_{v\in V}\mathcal{I}^k_v(y)$.
    Since $G$ is irreducible,
    $$
    \alpha_k=\max_{v\in V} \max_{y\in \Sigma_Y^k} n^{k\dim_B \pi(X)}\cdot \# \mathcal{I}^k_v(y).
    $$
    Noticing that $\dim_B \pi(X)=\lim_{k\to \infty} \frac{\log \# \pi \mathcal{I}^k}{k\log m}=\frac{h_{\textup{top}}(\pi \mathcal{I})}{\log m}$,  with Theorem \ref{th1} in hand,  it suffices to prove that
    \begin{equation}\label{e24}
    \sup_{y\in \pi \mathcal{I}} h_{\textup{top}}(\pi^{-1}(y))=\lim_{k\to \infty}\max_{v\in V}\max_{y\in \Sigma_Y^k} \frac{\log \#\mathcal{I}^k_v(y)}{k}.
    \end{equation}

    Since for $y\in \pi \mathcal{I}$,
    $$
    h_{\textup{top}}(\pi^{-1}(y))=\limsup_{k\to \infty}\frac{\log \# \mathcal{I}^k(y|_k)}{k}\leq \limsup_{k\to \infty}  \max_{v\in V} \frac{\log \# \mathcal{I}_v^k(y|_k)+\log \#V}{k},$$
    we have the direction $``\leq$'' in \eqref{e24}. On the other hand, let $\Delta$ be the right hand side
    of \eqref{e24}. Then for $\epsilon>0$, there exists $N$ in $\mathbb{N}$ such that for all $k$ in $\mathbb{N}$, there are $v_k,v'_k$ in $V$, $w^{(k)}$ in $E^{k+N}$ with $v_k\stackrel{w^{(k)}}{\longrightarrow} v_k'$ and $y^{(k)}=y_{w^{(k)}}$, such that
    \begin{equation}\label{e27}
    \log \# \mathcal{I}^{k+N}_{v_k,v'_k}(y^{(k)})\geq  \log \# \mathcal{I}^{k+N}_{v_k}(y^{(k)})-\log \#V\geq (k+N)(\Delta-\epsilon)-\log \#V.
    \end{equation}
    Noticing that  $G$ is irreducible, there exists $S$ in $\mathbb{N}$, such that for each distinct  pair $v,v'$ in $V$, there exists $w\in E^*$ with $|w|\leq S$ and $v\stackrel{w}{\to}v'$. For any $k$ in $\mathbb{N}$, pick a  directed path $\tilde{w}^{(k)}$  from  $v'_k$ to $v_{k+1}$ with length no more than $S$ if $v'_k\neq v_{k+1}$; pick $\tilde{w}^{(k)}=\emptyset$ if $v'_k=v_{k+1}$.  Write  $\tilde{y}^{(k)}=y_{\tilde{w}^{(k)}}$, $\dot{y}=y^{(1)}\tilde{y}^{(1)}y^{(2)}\tilde{y}^{(2)}\cdots \in \pi \mathcal{I}$, $\dot{y}^{(k)}=y^{(1)}\tilde{y}^{(1)}\cdots y^{(k)}\tilde{y}^{(k)}$ and $s_k=|\dot{y}^{(k)}|$. Noticing that  $\frac{k(k+1)}{2}+kN\leq s_k\leq \frac{k(k+1)}{2}+kN+kS $, by using \eqref{e27}, we have
    $$
    \begin{aligned}
    h_{\textup{top}}(\pi^{-1}(\dot{y}))&\geq \limsup_{k\to \infty} \frac{\log \# \mathcal{I}^{s_k}(\dot{y}^{(k)})}{s_k}\\
    &\geq \limsup_{k\to \infty} \frac{\log \#\mathcal{I}^{1+N}_{v_1,v'_1}(y^{(1)})+\cdots+\log \#\mathcal{I}^{k+N}_{v_k,v'_k}(y^{(k)}) }{s_k} \geq \Delta-\epsilon,
    \end{aligned}
    $$
    which gives the direction $``\geq $'' in \eqref{e24} by the arbitrary of $\epsilon.$
\end{proof}

\begin{lemma}\label{le6}
    When $G$ is irreducible, equality \eqref{e23} holds if and only if  $\dim_H X=\dim_B X.$
\end{lemma}
\begin{proof}
    \textit{The ``only if'' part.} Recall that from \cite[Corollary 3.2]{F23}, $\dim_H X =\dim_B X$ if and only if the measure of maximal entropy on $\mathcal{I}$ projects to the measure of maximal entropy on $\pi \mathcal{I}$, i.e. there exists a $\sigma$-invariant measure $\mu$ such that
    \begin{equation}\label{e28}
    h(\mu)=h_{\textup{top}}(\mathcal{I}) \quad \text{ and }\quad h(\mu\circ\pi^{-1})=h_{\textup{top}}(\pi\mathcal{I}),
    \end{equation}
    where $h(\mu)$ denotes the \textit{measure entropy} of $\mu$.

    It is due to Ledrappier and Walters \cite{LW77} that there is  a relative variational principle  for \eqref{e22} in the form
    \begin{equation}\label{ee}
    \sup_{m} h(m)=h(\nu)+\int_{\pi\mathcal{I}} h_{\textup{top}}(\pi^{-1}(y)) d\nu(y),
    \end{equation}
    where the supremum is taken over all the $\sigma$-invariant probability measure $m$ satisfying $m\circ \pi^{-1}=\nu.$ Choose  $\mu$ the measure satisfying $h(\mu)=h_{\textup{top}}(\mathcal{I})$. Let $\nu=\mu\circ\pi^{-1}$, then we have
    \begin{equation}\nonumber
    h_{\textup{top}}(\mathcal{I})=\sup_{m} h(m)=h(\nu)+\int_{\pi\mathcal{I}} h_{\textup{top}}(\pi^{-1}(y)) d\nu(y)\leq h_{\textup{top}}(\pi \mathcal{I})+\sup_{y\in \pi \mathcal{I}} h_{\textup{top}}(\pi^{-1}(y)).
    \end{equation}
    Combining this with \eqref{e23}, we immediately get $h(\mu \circ\pi^{-1})=h_{\textup{top}}(\pi \mathcal{I})$, which gives $\dim_H X=\dim_B X.$

    \vspace{0.2cm}

    \textit{The ``if'' part}. It follows from \cite[Theorem 3.1]{F24} ($\mathcal{I}$ is a subshift satisfying the so-called ``weak specification'' in \cite{F24} since $G$ is irreducible), there is a constant $C>0$ such that 
    $$
    C^{-1} \#\{x_{w}:w\in E^k,y_w=y'\}\leq  \#\{x_{w}:w\in E^k,y_w=y\}\leq C \#\{x_{w}:w\in E^k,y_w=y'\}
    $$
    for all $y,y'\in \pi \mathcal{I}^k$ and $k\geq 1.$ Therefore, for $y,y'\in \pi \mathcal{I}$,
    $$
    h_{\textup{top}}(\pi^{-1}(y))=h_{\textup{top}}(\pi^{-1}(y')).
    $$
    Since $\dim_H X =\dim_B X$, we can choose a $\sigma$-invariant measure $\mu$ satisfying \eqref{e28}. Let $\nu=\mu\circ \pi^{-1}$, by \eqref{ee}, we have 
    $$
    \begin{aligned}
    h_{\textup{top}}(\mathcal{I})&=\sup_{m:\  m\circ\pi^{-1}=\nu} h(m)=h(\nu)+\int_{\pi\mathcal{I}} h_{\textup{top}}(\pi^{-1}(y)) d\nu(y) \\
    &= h_{\textup{top}}(\pi \mathcal{I})+\sup_{y\in \pi \mathcal{I}} h_{\textup{top}}(\pi^{-1}(y)).
    \end{aligned}
    $$
    
\end{proof}
\begin{proof}[Proof of Theorem \ref{th2}]
    Combining Lemmas \ref{le5} and \ref{le6}, it suffices to prove that $\dim_B X=\dim_A X$ if and only if \eqref{e23} holds. Since $G$ is irreducible, the expression of $\dim_B X$ in \eqref{e25} degenerates to the  following form
    $$
    \dim_B X = \frac{h_{\textup{top}}(\mathcal{I})}{\log n}+  h_{\textup{top}}(\pi \mathcal{I})\left (\frac{1}{\log m}-\frac{1}{\log n}\right) .
    $$
    Combining this with Lemma \ref{le5} and \eqref{e23}, the theorem follows.
\end{proof}

\vspace{0.2cm}

Next, we consider the general case that $G$ may be not irreducible. Let $\{H_i=(V_i,E_i)\}_{i=1}^r$ be the collection of irreducible components of $G$. For $v,v'$ in $V_i$, it is easy to see that $\dim_B \pi(X_v)=\dim_B \pi(X_{v'})$. So, for $i=1,\cdots,r$, we can write $\lambda^{(i)}=\dim_B \pi(X_v)$ for $v\in V_i$.
For $k$ in $\mathbb{N}$, define
$$
\alpha_k^{(i)}=n^{k\lambda^{(i)}}\cdot \max_{v\in V_i} \max_{y\in \Sigma_Y^k}\#\{[w]\in [E^k]:i(w)=v, t(w)\in V_i,y_w=y \}.
$$

\begin{lemma}\label{le7}
    For each $i=1,\cdots,r $, the limit
    $\alpha^{(i)}:=\lim_{k\to \infty} (\alpha_k^{(i)})^{1/k}$ exists and
    $$\max_{i=1,\cdots,r} \alpha^{(i)}=\alpha.$$
\end{lemma}
\begin{proof}
    The existence of the limit of $\{(\alpha^{(i)}_k)^{1/k}\}_{k\geq 1}$ follows by Lemma \ref{lim}.

    %For $v\in \tilde{V}$, there exists $w\in E^*$ with $|w|\geq \#V$, such that $t(w)=v$. Noticing that $\{i(w_1),\cdots, i(w_{|w|}),v\}\subseteq V $, there exists $v'\in \cup_{i=1}^r V_i$ such that $v'\to v$.
    Note that $\dim_B \pi(X_{v})\geq \dim_B \pi(X_{v'})$ if $v\to v'.$ It is not hard to see that  there exists $C>0$ such that for $k$ in $\mathbb{N},$ $v$ in $\tilde{V}$, $y$ in $\Sigma_Y^k$,
    $$
    \sum_{[w]\in [E^k]:\atop i(w)=v,y_w=y} n^{k\eta(v,x_w,y_w)} \leq C\cdot\sum_{1\leq i_1,\cdots, i_j\leq r}\sum_{s_{i_1}+\cdots+s_{i_j}=k}\alpha_{s_{i_1}}^{(i_1)}\cdots \alpha_{s_{i_j}}^{(i_j)},
    $$
    which immediately yields that $\alpha\leq \max_{i=1,\cdots,r} \alpha^{(i)}$ (by a same argument considering the convergence radius of $\sum_{k\geq 1}\alpha_k t^k$
    in the proofs of \cite[Lemma 3.4]{F23} and \cite[Lemma 3.3]{KP96}).

    Conversely, since $\bigcup_{i=1}^r V_i\subseteq \tilde{V}$, for $i=1,\cdots,r$, we have
    $$
    \alpha_k\geq \frac{1}{\#\tilde{V}}\sum_{v\in \tilde{V}}\max_{y\in \Sigma_Y^k}n^{k\eta(v,x_w,y_w)}\geq \frac{1}{\#\tilde{V}}\alpha_k^{(i)}.
    $$
    The lemma follows.
\end{proof}

Recall that  $\{i\}^+$ denotes the collection of $1\leq j\leq r$ such that there is a path from a vertex in $H_i$ to a vertex in  $H_j$.

\begin{lemma}\label{le8}
    For $i=1,\cdots,r$, we have $\lambda^{(i)}=\max_{j\in \{i\}^+}\frac{h_{\textup{top}}(\pi \mathcal{I}_{H_j})}{\log m}.$
\end{lemma}
\begin{proof}
    It is same as \cite[Lemma 3.4]{F23}, which follows by a same argument in Lemma \ref{le7}.
\end{proof}

\begin{proof}[Proof of Theorem \ref{th3}]
   Firstly, combining Lemmas \ref{le7}, \ref{le8} and Theorem \ref{th2}, we have
    $$
    \begin{aligned}
    \dim_A X &=\max_{1\leq i\leq r} \left \{ \sup_{y\in \pi\mathcal{I}_{H_i}} \frac{h_{\textup{top}}(\pi^{-1}(y)\cap \mathcal{I}_{H_i})}{\log n} +\lambda^{(i)}\right \}
    \\
    &=\max_{1\leq i\leq r} \left \{ \sup_{y\in \pi\mathcal{I}_{H_i}} \frac{h_{\textup{top}}(\pi^{-1}(y)\cap \mathcal{I}_{H_i})}{\log n} +\max_{j\in \{i\}^+} \frac{h_{\textup{top}}(\pi\mathcal{I}_{H_j})}{\log m}\right \}.
    \end{aligned}
    $$

    Secondly, if \eqref{e26} holds, it is direct to see that $\dim_B X=\dim_A X$. Conversely, assume that  $\dim_B X=\dim_A X$. By \eqref{e25}, we can pick $(i,j)$ with $j\in \{i\}^+$ such that
    $$
    \dim_B X=\frac{h_{\textup{top}}(\mathcal{I}_{H_i})}{\log n}+ h_{\textup{top}}(\pi \mathcal{I}_{H_j})\left (\frac{1}{\log m}-\frac{1}{\log n}\right).
    $$
    Using \eqref{e22} and $\max_{j'\in \{i\}^+}h_{\textup{top}}(\pi \mathcal{I}_{H_{j'}})=h_{\textup{top}}(\pi \mathcal{I}_{H_j}) $ we get
    $$
    \dim_B X\leq \sup_{y\in \pi \mathcal{I}_{H_i}} \frac{h_{\textup{top}}(\pi^{-1}(y)\cap \mathcal{I}_{H_i})}{\log n} +\frac{h_{\textup{top}}(\pi \mathcal{I}_{H_j})}{\log m}+\frac{h_{\textup{top}}(\pi \mathcal{I}_{H_i})-h_{\textup{top}}(\pi \mathcal{I}_{H_j})}{\log n}\leq \dim_A X,
    $$
    which implies $h_{\textup{top}}(\pi \mathcal{I}_{H_{i}})=h_{\textup{top}}(\pi \mathcal{I}_{H_j})$ and  \eqref{e26} follows.

    Lastly, let $\{X_{v}^{(i)}\}_{v\in V_i}$ be the graph-directed Bedford-McMullen $(\times m,\times n)$-carpets family associated with $H_i$ and write $X^{(i)}=\bigcup_{v\in V_i} X_v^{(i)}.$ If $\eqref{e26}$ holds, by Lemma \ref{le6}, it holds that  $\dim_H X^{(i)}=\dim_B X^{(i)}$ for the same $i$ in the previous paragraph.
   Thus we have
    $$
    \dim_H X \geq \dim_H X^{(i)} =\dim_B X^{(i)} =\dim_B X,
    $$ giving that
    $\dim_H X=\dim_BX=\dim_AX.$
\end{proof}

\begin{proof}[Proof of Theorem \ref{th4}]
    The first part follows from Theorem \ref{th2}. For the second part, we consider the example of graph-directed Bedford-McMullen carpet family $\{X_a,X_b\}$ generated in the way illustrated in Figure \ref{figab}.
\begin{figure}[htp]
	\includegraphics[width=4.5cm]{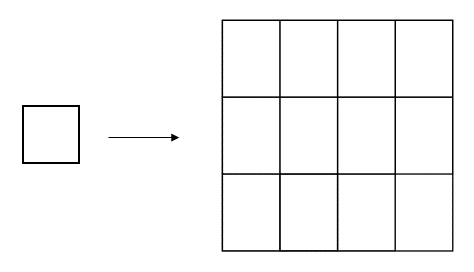}\hspace{1cm}
 \begin{picture}(0,0)
 \put(-149,35){a}
  \put(-49,35){a}
   \put(-64,35){a}
   \put(-64,54){a}
   \put(-64,14){a}
   \put(-49,54){a}
   \put(-49,14){a}
   \put(-95,35){b}
  \end{picture}
    \includegraphics[width=4.5cm]{a.png}
    \begin{picture}(0,0)
    \put(-121,34){b}
    \put(-67,34){b}
    \put(-67,54){b}
    \put(-67,14){b}
    \put(-51,54){b}
     \put(-35,54){b}
      \put(-19,54){b}
  \end{picture}
	\caption{An example of graph-directed Bedford-McMullen carpet family.}
	\label{figab}
\end{figure}

     At this time, $n=4$, $m=3$. It is not hard to check that $\dim_B \pi(X_{a})=\dim_B \pi(X_{b})=1$, and for $X=X_a\cup X_b$, $\dim_A X=2$  and $\dim_L X=1$, $\dim_B X=\dim_H X=\frac{3}{2}$.
\end{proof}

\section*{Acknowledgments}
The authors are grateful to Dr. Zhou Feng for sharing his latest paper \cite{F24}, which is useful for the proof of the Lemma \ref{le6}.

\bibliographystyle{amsplain}

\end{document}